\documentclass[11pt]{article}
\pdfoutput=1
\usepackage{amstext,amsmath,amssymb,amsthm,constants} 
\usepackage{latexsym}
\usepackage{exscale}
\usepackage{color}
\usepackage{float}

\usepackage[margin=1in]{geometry} 
\usepackage{amsmath,amsthm,amssymb, graphicx, multicol, array}

\usepackage{pgfplots}
\pgfplotsset{compat=1.6}

\pgfplotsset{soldot/.style={color=blue,only marks,mark=*}} \pgfplotsset{holdot/.style={color=blue,fill=white,only marks,mark=*}}

\RequirePackage{srcltx}

\numberwithin{equation}{section}
\RequirePackage{srcltx}

\numberwithin{equation}{section}

\newconstantfamily{smc}{symbol=c}
\newconstantfamily{lgc}{symbol=C}

\usepackage{mathtools}

\newtheorem{Th}{Theorem}[section]
\newtheorem{Lemma}[Th]{Lemma}

\theoremstyle{remark}
\newtheorem{Rem}{Remark}[section]
\newtheorem{Question}{Question}[section]

\def\XXint#1#2#3{{\setbox0=\hbox{$#1{#2#3}{\int}$ }
		\vcenter{\hbox{$#2#3$ }}\kern-.6\wd0}}

\title{Exceptional behavior in critical first-passage percolation and random sums}
\author{Michael Damron \thanks{Email: mdamron6@protonmail.com. The research of M. D. is supported by NSF grant DMS-2054559.} \\ \small{Georgia Tech}  \and Jack Hanson \thanks{Email: jhanson@ccny.cuny.edu. The research of J. H. is supported by NSF grant DMS-1954257, a PSC-CUNY grant and a CUNY JFRASE award via the CUNY Office of Research and the Sloan Foundation.}\\ \small{City College of New York} \and David Harper \thanks{Email: dharper40@gatech.edu. The research of D. H. is partially supported by NSF grant DMS-2054559.} \\ \small{Georgia Tech} \and Wai-Kit Lam\thanks{Email: waikitlam@ntu.edu.tw. The research of W.-K. L. is supported by the National Science and Technology Council in Taiwan Grant 110-2115-M002-012-MY3 and NTU New Faculty Founding Research Grant NTU-111L7452.} \\ \small{National Taiwan University}}

\begin{document}
	
	\maketitle 
	\begin{abstract}
We study first-passage percolation (FPP) on the square lattice. The model is defined using i.i.d.~nonnegative random edge-weights $(t_e)$ associated to the nearest neighbor edges of $\mathbb{Z}^2$. The passage time between vertices $x$ and $y$, $T(x,y)$, is the minimal total weight of any lattice path from $x$ to $y$. The growth rate of $T(x,y)$ depends on the value of $F(0) = \mathbb{P}(t_e=0)$: if $F(0) < 1/2$ then $T(x,y)$ grows linearly in $|x-y|$, but if $F(0) > 1/2$ then it is stochastically bounded. In the critical case, where $F(0) = 1/2$, $T(x,y)$ can be bounded or unbounded depending on the behavior of the distribution function $F$ of $t_e$ near 0. In this paper, we consider the critical case in which $T(x,y)$ is unbounded and prove the existence of an incipient infinite cluster (IIC) type measure, constructed by conditioning the environment on the event that the passage time from $0$ to a far distance remains bounded. This IIC measure is a natural candidate for the distribution of the weights at a typical exceptional time in dynamical FPP. A major part of the analysis involves characterizing the limiting behavior of independent nonnegative random variables conditioned to have small sum. We give conditions on random variables that ensure that such limits are trivial, and  several examples that exhibit nontrivial limits.
	\end{abstract}

	\section{Introduction}

\subsection{Background}

	In standard first-passage percolation (FPP) on $\mathbb{Z}^d$, nonnegative independent and identically distributed edge weights are assigned to the edges of the hypercubic lattice, and one considers properties of the resulting weighted graph metric. The model has generated much interest due to its mathematical challenges and connections \cite{C18,J00,R81} as well as the insights it provides into real-world phenomena, but many fundamental issues in FPP remain unresolved \cite{ADH17}. One important problem is the behavior of the model under perturbations: how does the structure of the metric space change under resampling of the random environment? The reader is directed to, for example, \cite{ADS23a,ADS23b,GH20}, for work on this question.

          In 2015, Ahlberg \cite{A15} introduced a dynamical version of the model, in which edge weights are resampled according to independent rate one Poisson processes, to study this question. Ahlberg was interested in possible ``exceptional'' behavior of the model, where the FPP metric $T$ behaves very differently at random times  than its almost sure behavior at deterministic times. Interesting exceptional behavior of this type is exhibited in a number of related models \cite{BHPS03,FNRS09,HPS97,SSS17}. Ahlberg showed that at the broadest relevant scale, no exceptional behavior is possible: almost surely, the linear growth rate of $T$ is the same at all times.

          However, in an important special case of the model, so-called critical FPP, the standard linear scaling fails to capture the leading-order growth rate of $T$. In critical FPP, the fraction of zero edge weights exactly equals $p_c$, the critical probability for Bernoulli percolation. In two-dimensional critical FPP, the asymptotic growth rate of $T$ is well-characterized \cite{DHL23,DLW17,Y18}. At one extreme, $T(0, x)$ can grow logarithmically in $|x|$; at the other, it can fail to diverge at all; it can also exhibit intermediate behavior, all depending on the shape of the distribution function of the edge weights near zero. In what follows, we focus on the case where $T(0, x)$ almost surely diverges as $|x| \to \infty$; a precise characterization appears below at \eqref{eq: rho_characterization}.

          The dynamical versions of such models of critical FPP exhibit exceptional times on constant scale --- namely, for $L\geq 0$, there almost surely exist random times at which the set $\{x: \, T(0, x) \leq L\}$ is infinite. The paper \cite{DHHL21} inaugurated the study of such exceptional times and showed that the set of exceptional times can display a wide range of behavior, depending on the underlying weight distribution. In particular, it is possible for the set of exceptional times to have upper Minkowski dimension $31/36$ or upper Minkowski dimension $1$, though its Hausdorff dimension is always $31/36$.

          The wide range of possible exceptional behavior leads naturally to the question of how this behavior is realized. What does the random environment look like at an exceptional time? The present work begins the study of this question. We construct a target distribution for this behavior by conditioning on $T$ metric balls being abnormally large: namely, conditioning that the passage time $T(0, \partial B(n))$ from $0$ to the boundary of $[-n, n]^2$ be smaller than some fixed $L \geq 0$ and taking $n \to \infty$. This is the FPP analogue of Kesten's incipient infinite cluster (IIC) measure \cite{kesteniic}. In Theorem~\ref{thm: main_result} below, we show that for any $L$, this measure is closely related to the IIC and characterize their relationship precisely.

          Our work involves the study of random sums, of independent but not necessarily i.i.d.~terms, conditioned to lie close to the infimum of their support. Our main result on such sums, which may be of independent interest, is Lemma~\ref{lem: resampling} below. It characterizes the distribution of a typical term of such a conditioned sum, under certain assumptions about the sequence of terms. More concretely, if $X_1, X_2, \dots$ are independent, nonnegative random variables and $L,\delta \geq 0$, the lemma gives conditions under which
          \[
          \mathbb{P}(X_1 \geq \delta \mid X_1 + \dots X_n \leq L) \to 0 \text{ as } n \to \infty.
          \]
          In the appendix, we study properties of such conditioned sums in detail. We show consequences of our Lemma~\ref{lem: resampling} for classes of such sums and also provide examples where the assumptions of the lemma do not hold, and where the distribution of a typical term looks very different from the conclusion of the lemma. We pose open problems relating to the distribution of terms of such conditioned random sums.



        \subsection{Formal definitions and main result}

	We now give formal definitions needed to state the main result. We will work on the two-dimensional square lattice with vertex set $\mathbb{Z}^2$ and its set $\mathbb{E}^2$ of nearest-neighbor edges. Let $(t_e)_{e \in \mathbb{E}^2}$ be an i.i.d.~family of nonnegative random variables with distribution function $F$. A path from a vertex $x$ to a vertex $y$ is an alternating sequence $x = x_0, e_0, x_1, e_2, \dots, x_{n-1}, e_{n-1}, x_n = y$ of vertices and edges such that for each $i = 0, \dots, n-1$, $e_i$ has endpoints $x_i$ and $x_{i+1}$. The passage time of a path $\gamma$ is $T(\gamma) = \sum_{i=0}^{n-1} t_{e_i}$, and the first-passage time between vertices $x$ and $y$ is
\[
T(x,y) = \inf_{\gamma : x \to y} T(\gamma),
\]
where the infimum is over all paths from $x$ to $y$.

For $n \geq 1$, consider the box $B(n) = [-n,n]^2$ with its vertex boundary $\partial B(n) = \{x \in B(n) : \|x\|_\infty = n\}$. We will be interested in the variable
\[
T(0,\partial B(n)) = \inf_{\gamma : 0  \to \partial B(n)} T(\gamma),
\]
where the infimum is over all paths starting at $0$ and ending at some vertex of $\partial B(n)$. The asymptotic behavior of $T(0,\partial B(n))$ is known to depend on the value of $F(0) = \mathbb{P}(t_e=0)$. Concretely, we have
\begin{align*}
\liminf_{n \to \infty} \frac{T(0,\partial B(n))}{n} > 0 \text{ a.s.} &\hspace{.3in}\text{if } F(0) < \frac{1}{2} \text{ and } \\
T(0,\partial B(n)) \text{ is a.s. bounded } &\hspace{.3in}\text{if } F(0) > \frac{1}{2}.
\end{align*}
In the ``critical'' case, where 
\begin{equation}\label{eq: percolation_condition}
F(0) = \frac{1}{2},
\end{equation}
the growth is more complicated and depends on the behavior of $F$ near 0. To describe it, let $F^{-1}$ be the generalized inverse
\[
F^{-1}(t) = \inf\{ x \in \mathbb{R} : F(x) \geq t\} \text{ for } t \in (0,1)
\]
and for $k \geq 2$, define $a_k = F^{-1}(1/2 + 1/2^k)$. Then \cite[Cor.~1.3]{DLW17} states that if we define the a.s.~monotone limit
\[
\rho = \lim_{n \to \infty} T(0,\partial B(n)),
\]
then
\begin{equation}\label{eq: rho_characterization}
\rho \text{ is }~\begin{cases}
\text{finite a.s.} &\quad\text{ if } \sum_{k=2}^\infty a_k < \infty \\
\text{infinite a.s.} &\quad\text{ if } \sum_{k=2}^\infty a_k = \infty.
\end{cases}
\end{equation}
Furthermore, if $\mathbb{E}t_e^{\eta} < \infty$ for some $\eta > 1/4$, then
\[
\mathbb{E} T(0,\partial B(n)) \asymp \sum_{k=2}^{\lfloor \log n \rfloor} a_k,
\]
meaning that the ratio of the left and right sides is bounded away from 0 and $\infty$ as $n \to \infty$.

We are interested in the critical case in which $\sum_{k=2}^\infty a_k = \infty$, so that $\rho = \infty$ a.s.. We wish to determine the conditional distribution of $(t_e)$ given the zero probability event that $\rho < \infty$. This distribution is the FPP analogue of the incipient infinite cluster (IIC) of Bernoulli percolation. As in that context, we phrase the question using a limiting procedure. Let $E \subset [0,\infty)^{\mathbb{E}^2}$ be a cylinder event---a product measurable set depending on finitely many coordinates. Our main result is that the limit
\[
\lim_{L \to \infty} \lim_{n \to \infty} \mathbb{P}((t_e) \in E \mid T(0,\partial B(n)) \leq L) \text{ exists}
\]
and is essentially Kesten's IIC measure. To define this measure, let $(t_e^B)$ be an i.i.d.~collection of Bernoulli$(1/2)$ random variables indexed by $\mathbb{E}^2$. Writing $T^B(\cdot, \cdot)$ for the corresponding passage time, it is known that $T^B(0,\partial B(n)) \to \infty$ a.s.~as $n \to \infty$ so that, in particular, $\mathbb{P}(T^B(0,\partial B(n)) = 0) \to 0$ as $n \to \infty$. If $E$ is a cylinder event, then Kesten's result \cite[Thm.~3]{kesteniic} states that the limit
\[
\nu(E) = \lim_{n \to \infty} \mathbb{P}((t_e^B) \in E \mid T^B(0,\partial B(n)) = 0) \text{ exists}
\]
and $\nu$ can be uniquely extended to a probability measure called the incipient infinite cluster. This $\nu$ is supported on configurations with $\rho = 0$. Analogously, we can define an IIC-type measure for first-passage percolation:
\begin{Lemma}\label{lem: fpp_iic}
Suppose that \eqref{eq: percolation_condition} holds. For any cylinder event  $E \subset [0,\infty)^{\mathbb{E}^2}$, the following limit exists:
\[
\widetilde{\nu}(E) = \lim_{n \to \infty} \mathbb{P}((t_e) \in E \mid T(0,\partial B(n))=0).
\]
If $E$ depends only on the variables $(\mathbf{1}_{\{t_e > 0\}})$, then $\widetilde{\nu}(E) = \nu(E)$.
\end{Lemma}
We give the proof of Lem.~\ref{lem: fpp_iic} after stating our main result. It shows that the IIC measure $\widetilde{\nu}$ for first-passage percolation is the distribution of the product $(\sigma_e s_e)$, where $(\sigma_e)$ is sampled from Kesten's IIC measure $\nu$, and $(s_e)$ is an independent family of random variables satisfying $\mathbb{P}(s_e \in A) = \mathbb{P}(t_e \in A \mid t_e > 0)$ for Borel sets $A$. Equivalently, it is the distribution of weights constructed as follows. First we sample a configuration $(\sigma_e)$ from $\nu$. Given $(\sigma_e)$, the weights at different edges $e$ are independent, equaling zero with probability one if $\sigma_e=0$, and having Radon-Nikodym derivative equal to $\mathbf{1}_{(0,\infty)}/\mathbb{P}(t_e > 0)$ relative to the original distribution of $(t_e)$ if $\sigma_e=1$. 

The simple form (and proof) of Lem.~\ref{lem: fpp_iic} is a direct result of essentially conditioning that $\rho = 0$. This turns the question of the existence of the limit into a Bernoulli percolation problem, which was solved by Kesten. Our main result, Thm.~\ref{thm: main_result}, instead conditions that $\rho < \infty$, and this event, while containing the event $\{\rho = 0\}$, is significantly more complicated. In particular, the former event contains many weight configurations in which optimal paths between distant points take a large number of nonzero edge weights. For this reason, it is not even a priori clear that the limit in Thm.~\ref{thm: main_result} should be related to the IIC, and we will need to use the machinery developed for critical FPP in recent years to give the proof.


\begin{Th}\label{thm: main_result}
Suppose that \eqref{eq: percolation_condition} holds and $\sum_{k=2}^\infty a_k = \infty$. Let $E \subset [0,\infty)^{\mathbb{E}^2}$ be a cylinder event. For any $L \geq 0$,
\[
\lim_{n \to \infty} \mathbb{P}((t_e) \in E \mid T(0,\partial B(n)) \leq L) = \widetilde{\nu}(E).
\]
\end{Th}

Our proof of Thm.~\ref{thm: main_result} will be broken into three steps. In Sec.~\ref{sec: flibbertygibbet}, we reduce the problem, showing that it suffices to consider cylinders $E$ of the form $\{T(0,\partial B(K)) \geq \delta\}$ for fixed $K,\delta>0$. In Sec.~\ref{sec: circuit_construction}, we use a a circuit construction to conditionally decompose $T(0,\partial B(n))$ into a sum of independent variables. In Sec.~\ref{sec: resampling_section}, we prove Lem.~\ref{lem: resampling}, our main tool about nonnegative independent random variables conditioned to have small sum, and apply it to the decomposition of $T(0,\partial B(n))$.

\begin{Rem}
The proof of Thm.~\ref{thm: main_result} does not use any particular properties of the square lattice and is valid on a wide range of planar graphs including site and edge FPP on the triangular, square, and hexagonal lattices. The main requirement is that Kesten's IIC measure exists, and this was proved in \cite{kesteniic} for a general class of periodic graphs $G$ imbedded in the plane with corresponding matching graph $G^\ast$. See \cite[Sec.~2.1, 2.2]{kestenbook} for a discussion of such graphs.
\end{Rem}

We finish this section with the proof of Lem.~\ref{lem: fpp_iic}.
\begin{proof}
We represent the weights $t_e$ as a product of independent random variables $t_e = t_e^Bs_e$, where $(t_e^B) = (\mathbf{1}_{\{t_e > 0\}})$ is an i.i.d.~collection of Bernoulli$(1/2)$ random variables, and $(s_e)$ is an i.i.d.	~collection of variables satisfying $\mathbb{P}(s_e \in A) = \mathbb{P}(t_e \in A \mid t_e>0)$ for Borel sets $A$. Consider an event $D$ of the form $D = D_1 \times D_2$, where $D_i \in [0,\infty)^{\mathbb{E}^2}$ are cylinder events for $i=1,2$. Then
\begin{align}
\mathbb{P}(((t_e^B),(s_e)) \in D \mid T(0,\partial B(n)) = 0) &= \mathbb{P}((t_e^B) \in D_1, (s_e) \in D_2 \mid T^B(0,\partial B(n))=0) \nonumber \\
&= \mathbb{E}\left( \mathbb{P}\left( (t_e^B) \in D_1,(s_e) \in D_2 \mid (t_e^B)\right) \mid T^B(0, \partial B(n))=0\right) \nonumber \\
&= \mathbb{P}((s_e) \in D_2) \mathbb{P}((t_e^B) \in D_1 \mid T^B(0,\partial B(n))=0) \label{eq: to_generalize} \\
&\to \nu(D_1) \mathbb{P}((s_e) \in D_2)  \text{ as } n \to \infty. \nonumber
\end{align}
Therefore the pair $((t_e^B),(s_e))$ under $\mathbb{P}(\cdot \mid T(0,\partial B(n)) =0)$ converges in distribution to $((\sigma_e),(s_e))$, where $(\sigma_e)$ is an independent configuration sampled from the IIC measure $\nu$. Because the mapping $((t_e^B),(s_e)) \mapsto (t_e^Bs_e)$ is continuous in the product topology, the continuous mapping theorem gives that the distribution of $(t_e)$ under $\mathbb{P}(\cdot \mid T(0,\partial B(n)) =0)$ converges to that of $(\sigma_e s_e)$. Thus the limit defining $\widetilde{\nu}(E)$ exists and equals $\mathbb{P}((\sigma_e s_e) \in E)$. Taking $D_2$ to be the sure event shows that $\widetilde{\nu}(D_1) = \nu(D_1)$, implying the second statement of the lemma.
\end{proof}

\section{Proof of Thm.~\ref{thm: main_result}}\label{sec: main_proof}

For the remainder, we let $(t_e)$ be such that \eqref{eq: percolation_condition} holds and $\sum_{k=2}^\infty a_k = \infty$. The proof of Thm.~\ref{thm: main_result} will be broken into a few steps. 

\subsection{Step 1: problem reduction}\label{sec: flibbertygibbet}
First we give a simpler condition that implies the theorem. For any $L \geq 0$, integer $K \geq 0$, and $\eta>1/2$, suppose that we can prove that
\begin{equation}\label{eq: simpler_condition}
\lim_{n \to \infty} \mathbb{P}\left(T(0,\partial B(K)) \geq F^{-1}(\eta) \mid T(0,\partial B(n)) \leq L\right) = 0.
\end{equation}
In the case that $\lim_{\eta \downarrow 1/2} F^{-1}(\eta) = 0$ (there is no ``gap'' between 0 and the rest of the support of $t_e$), this statement is equivalent to $\lim_{n \to \infty} \mathbb{P}(T(0,\partial B(K)) \geq \delta \mid T(0,\partial B(n)) \leq L) = 0$ for all $\delta>0$. The above formulation is more convenient for simultaneously dealing with the case that $\lim_{\eta \downarrow 1/2} F^{-1}(\eta) > 0$. 

We will now argue that \eqref{eq: simpler_condition} implies the result of Thm.~\ref{thm: main_result}, so suppose that \eqref{eq: simpler_condition} holds. Let $E \subset [0,\infty)^{\mathbb{E}^2}$ be a cylinder event, and let $N$ be such that $E$ depends on the weights for edges in $B(N)$. For $K_2 > K_1$, let $\mathsf{F}(K_1,K_2)$ be the event that the annulus $\text{Ann}(K_1,K_2) = B(K_2) \setminus B(K_1)$ contains a circuit (a path whose initial and final points coincide) around the origin with total weight zero. Then for $L \geq 0$ and $n \geq 1$,
\begin{align}
\mathbb{P}((t_e) \in E \mid T(0,\partial B(n)) \leq L) &= \mathbb{P}((t_e) \in E, \mathsf{F}(\sqrt{K},K)^c \mid T(0,\partial B(n)) \leq L) \label{eq: term_1_breakdown}\\
&+ \mathbb{P}((t_e) \in E, \mathsf{F}(\sqrt{K},K) \mid T(0,\partial B(n)) \leq L) \nonumber.
\end{align}
The event $\mathsf{F}(\sqrt{K},K)^c$ is increasing in the weights $(t_e)$, whereas $\{T(0,\partial B(n)) \leq L\}$ is decreasing, so the Harris inequality \cite{H60} gives
\[
\text{RHS of }\eqref{eq: term_1_breakdown} \leq \mathbb{P}(\mathsf{F}(\sqrt{K},K)^c \mid T(0,\partial B(n)) \leq L) \leq \mathbb{P}(\mathsf{F}(\sqrt{K},K)^c).
\]
By \cite[(11.72)]{grimmettbook}, there exists $\Cl[smc]{c: F_constant}>0$ such that $\mathbb{P}(\mathsf{F}(\ell,2\ell)) \geq \Cr{c: F_constant}$ for all $\ell \geq 1$. By independence of the weights in disjoint annuli, there exists $\Cl[lgc]{c: new_F_constant_1}>0$ and $\Cl[smc]{c: new_F_constant_2}>0$ such that 
\begin{equation}\label{eq: F_bound}
\mathbb{P}(\mathsf{F}(\sqrt{K},K)^c) \leq \Cr{c: new_F_constant_1}K^{-\Cr{c: new_F_constant_2}}.
\end{equation}
Putting this in \eqref{eq: term_1_breakdown} gives
\begin{equation}\label{eq: put_a_circuit}
|\mathbb{P}((t_e) \in E \mid T(0,\partial B(n)) \leq L) - \mathbb{P}((t_e) \in E, \mathsf{F}(\sqrt{K},K) \mid T(0,\partial B(n)) \leq L)| \leq \Cr{c: new_F_constant_1} K^{-\Cr{c: new_F_constant_2}}.
\end{equation}

Any circuit $\mathcal{C}$ around the origin in $\text{Ann}(\sqrt{K},K)$, viewed as a plane curve, splits the plane into at least two connected components. One of these components, $\mathcal{R}(\mathcal{C})$, is unbounded. We say that $\mathcal{C}$ is the outermost zero-weight circuit around the origin in $\text{Ann}(\sqrt{K},K)$ if all edges $e$ on $\mathcal{C}$ have $t_e=0$ and $\mathcal{R}(\mathcal{C})$ is minimal. It is well-known (see \cite[p.~317]{grimmettbook}) that if there is a zero-weight circuit around the origin in $\text{Ann}(\sqrt{K},K)$, then there is a unique outermost vertex self-avoiding one. For any deterministic circuit $\mathcal{C}$, let $\mathsf{F}(\mathcal{C})$ be the event that $\mathcal{C}$ is the outermost vertex self-avoiding zero-weight circuit around the origin in $\text{Ann}(\sqrt{K},K)$. This event depends only on the weights for edges on $\mathcal{C}$ or in its exterior $\mathcal{R}(\mathcal{C})$. Then
\begin{align}
&\mathbb{P}((t_e) \in E, \mathsf{F}(\sqrt{K},K) \mid T(0,\partial B(n)) \leq L) \nonumber\\
=~& \sum_\mathcal{C} \mathbb{P}((t_e) \in E, \mathsf{F}(\mathcal{C}) \mid T(0,\partial B(n)) \leq L) \nonumber \\
=~& \sum_\mathcal{C} \mathbb{P}((t_e) \in E, \mathsf{F}(\mathcal{C}),T(0,\mathcal{C}) = 0 \mid T(0,\partial B(n)) \leq L) \nonumber \\
+& \sum_\mathcal{C} \mathbb{P}((t_e) \in E, \mathsf{F}(\mathcal{C}) , T(0,\mathcal{C}) > 0 \mid T(0,\partial B(n)) \leq L). \label{eq: to_put_assumption_in}
\end{align}
Here, $T(0,\mathcal{C})$ is defined as the minimal value of $T(\gamma)$ over all paths $\gamma$ that start at 0 and end at some vertex of $\mathcal{C}$. We now show that because of our assumption \eqref{eq: simpler_condition}, the term in \eqref{eq: to_put_assumption_in} is negligible.

For any $\eta>1/2$, \eqref{eq: to_put_assumption_in} is no larger than
\begin{align}
&\sum_{\mathcal{C}} \mathbb{P}(\mathsf{F}(\mathcal{C}), T(0,\mathcal{C}) \geq F^{-1}(\eta) \mid T(0,\partial B(n)) \leq L) \nonumber \\
+~& \sum_\mathcal{C} \mathbb{P}(\mathsf{F}(\mathcal{C}), T(0,\mathcal{C}) \in  (0,F^{-1}(\eta)) \mid T(0,\partial B(n)) \leq L) \nonumber \\
\leq~& \mathbb{P}(T(0,\partial B(K)) \geq F^{-1}(\eta) \mid T(0,\partial B(n)) \leq L) \nonumber \\
+~& \sum_{\mathcal{C}} \mathbb{P}(\mathsf{F}(\mathcal{C}), T(0,\mathcal{C}) \in (0,F^{-1}(\eta)) \mid T(0,\partial B(n)) \leq L). \label{eq: davids_term}
\end{align} 
On the event $\mathsf{F}(\mathcal{C})$, the circuit $\mathcal{C}$ has total weight zero, so $T(0,\partial B(n)) = T(0,\mathcal{C}) + T(\mathcal{C}, \partial B(n))$, assuming $n \geq K$. Furthermore, $\mathsf{F}(\mathcal{C})$ and $T(\mathcal{C},\partial B(n))$ depend on weights on $\mathcal{C}$ or in its exterior, whereas $T(0,\mathcal{C})$ depends on weights in the interior of $\mathcal{C}$, so they are independent. Using this, the summand in \eqref{eq: davids_term} is no larger than
\begin{align*}
&\frac{\mathbb{P}(\mathsf{F}(\mathcal{C}), T(0,\mathcal{C}) \in (0,F^{-1}(\eta)), T(\mathcal{C},\partial B(n)) \leq L)}{\mathbb{P}(\mathsf{F}(\mathcal{C}), T(0,\mathcal{C}) + T(\mathcal{C},\partial B(n)) \leq L)}  \mathbb{P}(\mathsf{F}(\mathcal{C}) \mid T(0,\partial B(n)) \leq L) \\
\leq~&\frac{\mathbb{P}(T(0,\mathcal{C}) \in (0,F^{-1}(\eta)))\mathbb{P}(\mathsf{F}(\mathcal{C}), T(\mathcal{C},\partial B(n)) \leq L)}{\mathbb{P}(T(0,\mathcal{C})=0, \mathsf{F}(\mathcal{C}), T(\mathcal{C},\partial B(n)) \leq L)}  \mathbb{P}(\mathsf{F}(\mathcal{C}) \mid T(0,\partial B(n)) \leq L) \\
=~&\frac{\mathbb{P}(T(0,\mathcal{C})\in (0,F^{-1}(\eta)))}{\mathbb{P}(T(0,\mathcal{C})=0)} \mathbb{P}(\mathsf{F}(\mathcal{C}) \mid T(0,\partial B(n)) \leq L) \\
\leq~& \frac{2(2K+1)^2 \left(\eta-\frac{1}{2}\right)}{\mathbb{P}(T(0,\partial B(K))=0)} \mathbb{P}(\mathsf{F}(\mathcal{C}) \mid T(0,\partial B(n)) \leq L) \\
\leq~& \frac{2(2K+1)^2 \left( \eta - \frac{1}{2}\right)}{2^{-2(2K+1)^2}} \mathbb{P}(\mathsf{F}(\mathcal{C}) \mid T(0,\partial B(n)) \leq L).
\end{align*}
In the second to last inequality we have used that if $T(0,\mathcal{C}) \in (0,F^{-1}(\eta))$, then there must be an edge $e$ in $B(K)$ with $t_e \in (0,F^{-1}(\eta))$. Placing this back in \eqref{eq: davids_term}, we obtain that \eqref{eq: to_put_assumption_in} is no larger than 
\[
 \mathbb{P}\left(T(0,\partial B(K)) \geq F^{-1}(\eta) \mid T(0,\partial B(n)) \leq L\right) + \frac{2(2K+1)^2\left(\eta-\frac{1}{2}\right)}{2^{-2(2K+1)^2}}.
 \]
This can be made small by first choosing $\eta$ close to $1/2$, and then taking $n \to \infty$, using our assumption \eqref{eq: simpler_condition}. In total, if we put this back into \eqref{eq: to_put_assumption_in}, we obtain for any $K\geq 1$ and $L \geq 0$
\begin{align}
\mathbb{P}((t_e) \in E, \mathsf{F}(\sqrt{K},K) \mid T(0,\partial B(n)) \leq L) - \sum_\mathcal{C} \mathbb{P}((t_e) \in E, \mathsf{F}(\mathcal{C}) &, T(0,\mathcal{C}) = 0 \mid T(0,\partial B(n)) \leq L)   \nonumber \\
&\to 0 \text{ as } n \to \infty. \label{eq: ingredient_2}
\end{align}

To deal with the right term in \eqref{eq: ingredient_2}, we take $K > N^2$. We again use that if $n \geq K$, $T(0,\partial B(n)) = T(0,\mathcal{C}) + T(\mathcal{C},\partial B(n))$ on $\mathsf{F}(\mathcal{C})$ to get
\begin{align}
&\sum_\mathcal{C} \mathbb{P}((t_e) \in E, \mathsf{F}(\mathcal{C}), T(0,\mathcal{C}) = 0 \mid T(0,\partial B(n)) \leq L)\nonumber \\
=~& \sum_\mathcal{C} \frac{\mathbb{P}((t_e) \in E, T(0,\mathcal{C})=0,  \mathsf{F}(\mathcal{C}),T(\mathcal{C},\partial B(n)) \leq L)}{\mathbb{P}(T(0,\partial B(n)) \leq L)} \nonumber \\
=~& \sum_\mathcal{C} \frac{\mathbb{P}((t_e) \in E, T(0,\mathcal{C})=0)\mathbb{P}(\mathsf{F}(\mathcal{C}), T(\mathcal{C},\partial B(n)) \leq L)}{\mathbb{P}(T(0,\partial B(n)) \leq L)}. \label{eq: im_on_a_plane}
\end{align}
In the second equality we have used independence in a similar way as before: $\mathsf{F}(\mathcal{C})$ and $\{T(\mathcal{C}, \partial B(n)) \leq L\}$ depend only on the weights on $\mathcal{C}$ and in its exterior, whereas $E$ and $\{T(0, \mathcal{C}) = 0\}$ depend only on the weights in the interior of $\mathcal{C}$. From \eqref{eq: put_a_circuit}, \eqref{eq: ingredient_2}, and \eqref{eq: im_on_a_plane}, we obtain the following statement: if $\epsilon>0$, for sufficiently large (fixed) $K \in (N^2,\infty)$ and all large $n$,
\begin{equation}\label{eq: to_use_later_my_man}
\left| \mathbb{P}((t_e) \in E \mid T(0,\partial B(n)) \leq L)   - \sum_\mathcal{C} \frac{\mathbb{P}((t_e) \in E, T(0,\mathcal{C})=0)\mathbb{P}(\mathsf{F}(\mathcal{C}), T(\mathcal{C},\partial B(n)) \leq L)}{\mathbb{P}(T(0,\partial B(n)) \leq L)}\right| \leq \epsilon.
\end{equation}

We next claim that if $K$ is large enough, then for any deterministic circuit $\mathcal{C}$ around the origin in $\text{Ann}(\sqrt{K},K)$,
\begin{equation}\label{eq: extended_kesten}
(1-\epsilon) \widetilde{\nu}(E) \leq \mathbb{P}((t_e) \in E \mid T(0,\mathcal{C})=0) \leq (1+\epsilon) \widetilde{\nu}(E).
\end{equation}
This follows from almost the same argument as that for Lem.~\ref{lem: fpp_iic}. Indeed, represent $t_e$ as a product of independent random variables $t_e = t_e^Bs_e$, where $t_e^B = \mathbf{1}_{\{t_e > 0\}}$ has Bernoulli$(1/2)$ distribution and $s_e$ satisfies $\mathbb{P}(s_e \in A) = \mathbb{P}(t_e \in A \mid t_e>0)$ for Borel sets $A$. If $D_1,D_2 \subset [0,\infty)^{\mathbb{E}^2}$ are cylinder events, then the same argument establishing \eqref{eq: to_generalize} gives
\begin{equation}\label{eq: hamburger_helper}
\mathbb{P}\left((t_e^B) \in D_1, (s_e) \in D_2 \mid T(0,\mathcal{C}) = 0\right) = \mathbb{P}((s_e) \in D_2) \mathbb{P}\left((t_e^B) \in D_1 \mid T^B(0,\mathcal{C})=0\right).
\end{equation}
By \cite[Eq.~(3.27)]{jarai03}, one has
\[
\lim_{m \to \infty \atop \mathcal{C} \text{ surrounds }B(m)} \mathbb{P}\left((t_e^B) \in D_1 \mid T^B(0,\mathcal{C})=0\right) = \nu(D_1).
\]
By \eqref{eq: hamburger_helper}, we have
\[
\lim_{m \to \infty \atop \mathcal{C} \text{ surrounds } B(m)} \mathbb{P}((t_e^B) \in D_1, (s_e) \in D_2 \mid T(0,\mathcal{C}) = 0)=  \mathbb{P}((s_e) \in D_2) \nu(D_1).
\]
Thus the pair $((t_e^B),(s_e))$ converges in distribution as $m \to \infty$ under the measure $\mathbb{P}(\cdot \mid T(0,\mathcal{C})=0)$ to that of $((\sigma_e),(s_e))$, where $(\sigma_e)$ has distribution $\nu$, and $(\sigma_e)$ and $(s_e)$ are independent. Again, by the continuous mapping theorem, $(t_e)$ converges in distribution as $m \to \infty$ under the measure $\mathbb{P}(\cdot \mid T(0,\mathcal{C})=0)$ to that of $(\sigma_e s_e)$. This establishes \eqref{eq: extended_kesten}.

Putting \eqref{eq: extended_kesten} into \eqref{eq: im_on_a_plane}, we find
\begin{align}
& (1-\epsilon) \widetilde{\nu}(E) \sum_\mathcal{C} \frac{\mathbb{P}(T(0,\mathcal{C})=0)\mathbb{P}(\mathsf{F}(\mathcal{C}), T(\mathcal{C},\partial B(n)) \leq L)}{\mathbb{P}(T(0,\partial B(n)) \leq L)} \nonumber \\
\leq~& \sum_\mathcal{C} \frac{\mathbb{P}((t_e) \in E, T(0,\mathcal{C})=0)\mathbb{P}(\mathsf{F}(\mathcal{C}), T(\mathcal{C},\partial B(n)) \leq L)}{\mathbb{P}(T(0,\partial B(n)) \leq L)} \nonumber \\
\leq~&(1+\epsilon) \widetilde{\nu}(E) \sum_\mathcal{C} \frac{\mathbb{P}(T(0,\mathcal{C})=0)\mathbb{P}(\mathsf{F}(\mathcal{C}), T(\mathcal{C},\partial B(n)) \leq L)}{\mathbb{P}(T(0,\partial B(n)) \leq L)}. \label{eq: last_to_do_my_friend}
\end{align}
Taking $E$ to be the sure event in \eqref{eq: to_use_later_my_man} produces
\[
1- \epsilon \leq \sum_\mathcal{C} \frac{\mathbb{P}(T(0,\mathcal{C})=0)\mathbb{P}(\mathsf{F}(\mathcal{C}), T(\mathcal{C},\partial B(n)) \leq L)}{\mathbb{P}(T(0,\partial B(n)) \leq L)} \leq 1+\epsilon.
\]
We put this in \eqref{eq: last_to_do_my_friend} to obtain
\[
(1-\epsilon)^2 \widetilde{\nu}(E) \leq \sum_\mathcal{C} \frac{\mathbb{P}((t_e) \in E, T(0,\mathcal{C})=0)\mathbb{P}(\mathsf{F}(\mathcal{C}), T(\mathcal{C},\partial B(n)) \leq L)}{\mathbb{P}(T(0,\partial B(n)) \leq L)} \leq (1+\epsilon)^2 \widetilde{\nu}(E).
\]
Last, returning to \eqref{eq: to_use_later_my_man}, we find for large $n$ that
\[
(1-\epsilon)^2 \widetilde{\nu}(E) - \epsilon \leq \mathbb{P}((t_e) \in E \mid T(0,\partial B(n)) \leq L) \leq (1+\epsilon)^2 \widetilde{\nu}(E)+ \epsilon.
\]
Since $\epsilon$ is arbitrary, we have now shown the statement of Thm.~\ref{thm: main_result}, assuming that \eqref{eq: simpler_condition} holds.

\subsection{Step 2: circuit construction and independence properties}\label{sec: circuit_construction}

We are left to verify \eqref{eq: simpler_condition}, so from now on, we fix $L \geq 0$, an integer $K \geq 0$, and $\eta>1/2$. In this step, we use a circuit construction to represent $T(0,\partial B(n))$ (conditionally) as a sum of independent random variables, so that in the next step, we can apply Lem.~\ref{lem: resampling}, our general result about the distribution of nonnegative random variables conditioned to have a small sum.

Let $R>0$ be an integer to be taken large at the end of the proof. We begin by defining events $(E_k)_{k \geq 1}$ as
\begin{align*}
E_k &= \mathsf{F}\left(R^{3k},R^{3k+1}\right) \cap \mathsf{F}\left( R^{3k+2},R^{3k+3}\right)\\
&=  \bigg\{ 
\begin{array}{c}
\text{Ann}\left(R^{3k},R^{3k+1}\right) \text{ and } \text{Ann}\left( R^{3k+2},R^{3k+3}\right) \text{ contain} \\
\text{circuits around the origin with total weight zero}
\end{array}
\bigg\}.
\end{align*}
Here, $\mathsf{F}$ is the event defined above \eqref{eq: term_1_breakdown}. 
By \eqref{eq: F_bound}, if $R$ is large enough, we have
\begin{equation}\label{eq: E_k_lower_bound}
\mathbb{P}(E_k) \geq 1 - 2\Cr{c: new_F_constant_1} R^{-\Cr{c: new_F_constant_2}} \geq \frac{1}{2},
\end{equation}
so the Borel-Cantelli lemma and independence imply that a.s., infinitely many $E_k$ occur. We enumerate (most of) the $k$ for which $E_k$ occurs as follows. Let $\kappa_1 = \min\{k \geq K : E_k \text{ occurs}\}.$ For $i \geq 1$, let
\[
\kappa_{i+1} = \min\{ k \geq \kappa_i + 1 : E_k \text{ occurs}\}.
\]

\begin{figure}
		\centering
		\includegraphics[width=0.8\linewidth, trim={0 0 0 15cm}, clip]{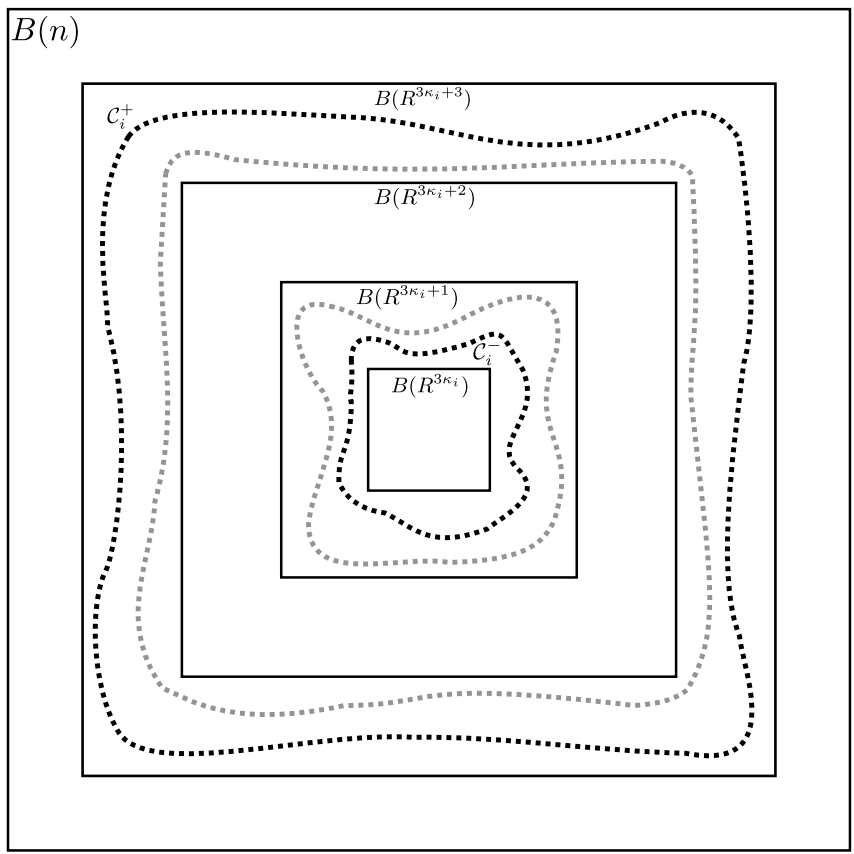}
		\caption{Depiction of the circuit construction in step 2. The annulus $\mathrm{Ann}(R^{3\kappa_i}, R^{3\kappa_i+3})$ is shown within the box $B(n)$, and it is split into four annuli. Because the event $E_{\kappa_i}$ occurs, we can find zero-weight circuits around the origin in the subannuli $\mathrm{Ann}(R^{3\kappa_i}, R^{3\kappa_i+1})$ and $\mathrm{Ann}(R^{3\kappa_i+2}, R^{3\kappa_i+3})$. The circuit $\mathcal{C}_i^-$ is the innermost in the first subannulus and $\mathcal{C}_i^+$ is the outermost in the second subannulus. They grey circuits are typical circuits that are neither innermost nor outermost and are drawn for reference.}
		\label{fig: circuits}
	\end{figure}
	
Next, we define circuits $\mathcal{C}^-_i$ and $\mathcal{C}^+_i$. Let $\mathcal{C}_1^- = \{0\}$ be the trivial circuit with no edges and only one vertex, the origin. For $i \geq 2$, let $\mathcal{C}_i^-$ be the innermost zero-weight circuit around the origin in $\text{Ann}\left( R^{3\kappa_i}, R^{3\kappa_i+1}\right)$. For $i \geq 1$, let $\mathcal{C}_i^+$ be the outermost zero-weight circuit around the origin in $\text{Ann}\left( R^{3\kappa_i+2}, R^{3\kappa_i+3}\right)$. For $n \geq 1$, there will be a largest such circuit that remains in $B(n)$: set $\mathcal{I} = \mathcal{I}(n)$ to be
\[
\mathcal{I} = \max\left\{ i \geq 1 : R^{3\kappa_i+3} \leq n \right\}.
\]
If the set is empty, we define $\mathcal{I}=0$. (See Fig.~\ref{fig: circuits}.) Because these circuits do not contain positive-weight edges, we have
\begin{equation}\label{eq: T_decomposition}
T(0,\partial B(n)) = T_1^- + T_1^+ + \dots + T_{\mathcal{I}}^- + T_{\mathcal{I}}^+ + R_n \text{ if } \mathcal{I} \geq 1,
\end{equation}
where $T_1^- = 0$, for $i \geq 2$, $T_i^- = T(\mathcal{C}_{i-1}^+, \mathcal{C}_i^-)$, and for $i \geq 1$, $T_i^+ = T(\mathcal{C}_i^-, \mathcal{C}_i^+)$. Last, the remainder term $R_n = T(\mathcal{C}_{\mathcal{I}}^+, \partial B(n))$.


We first show in \eqref{eq: independence} that, conditional on the circuits, the passage times $T_i^-,T_i^+,R_n$ are all independent. Furthermore, if $\kappa_1=K$, this conditioning gives no information about the region between circuits $\mathcal{C}_i^-$ and $\mathcal{C}_i^+$. Therefore (see \eqref{eq: free_percolation}) for deterministic $\mathsf{C}_i^-,\mathsf{C}_i^+$, this conditional distribution of $T_i^+$ when $\mathcal{C}_i^- = \mathsf{C}_i^-$ and $\mathcal{C}_i^+ = \mathsf{C}_i^+$ is the same as the (unconditional) distribution of $T(\mathsf{C}_i^-,\mathsf{C}_i^+)$. (This is true for $i \geq 2$ even when $\kappa_1> K$.)
\begin{Lemma}\label{lem: independence}
Let $I \geq 1$, $n \geq 1$, and $\mathsf{C}_1^-,\mathsf{C}_1^+, \dots, \mathsf{C}^-_I, \mathsf{C}_I^+$ be deterministic circuits such that $\mathsf{C}_1^+ \subset \mathrm{Ann}(R^{3K+2}, R^{3K+3})$ and $\mathbb{P}\left( \mathcal{I} = I, \cap_{i=1}^I \{\mathcal{C}_i^- = \mathsf{C}_i^-, \mathcal{C}_i^+ = \mathsf{C}_i^+\}\right) >0$. For any Borel sets $A,A_1^-, A_1^+, \dots, A_I^-, A_I^+ \subset \mathbb{R}$, we have
\begin{equation}\label{eq: independence}
\begin{split}
&\mathbb{P}\left( \cap_{i=1}^I \left\{ T_i^- \in A_i^-, T_i^+ \in A_i^+\right\}, R_n \in A \mid \mathcal{I} = I, \cap_{i=1}^I \{\mathcal{C}_i^- = \mathsf{C}_i^-, \mathcal{C}_i^+ = \mathsf{C}_i^+\} \right)  \\
=~& \prod_{i=1}^I \mathbb{P}\left( T_i^- \in A_i^- \mid \mathcal{I} = I, \cap_{i=1}^I \{\mathcal{C}_i^- = \mathsf{C}_i^-, \mathcal{C}_i^+ = \mathsf{C}_i^+\}\right)  \\
\times~& \prod_{i=1}^I \mathbb{P}\left( T_i^+ \in A_i^+ \mid \mathcal{I} = I, \cap_{i=1}^I \{\mathcal{C}_i^- = \mathsf{C}_i^-, \mathcal{C}_i^+ = \mathsf{C}_i^+\}\right) \\
\times~&  \mathbb{P}\left( R_n \in A \mid \mathcal{I} = I, \cap_{i=1}^I \{\mathcal{C}_i^- = \mathsf{C}_i^-, \mathcal{C}_i^+ = \mathsf{C}_i^+\}\right) 
\end{split}
\end{equation}
Furthermore, for each $i=1, \dots, I$,
\begin{equation}\label{eq: free_percolation}
\mathbb{P}\left( T_i^+ \in A_i^+ \mid \mathcal{I}=I, \cap _{i=1}^I \{\mathcal{C}_i^-=\mathsf{C}_i^-, \mathcal{C}_i^+ = \mathsf{C}_i^+\}\right) = \mathbb{P}(T(\mathsf{C}_i^-, \mathsf{C}_i^+) \in A_i^+)
\end{equation}
\end{Lemma}
\begin{proof}
Fix $k_1, \dots, k_I$ such that for each $i=1, \dots, I$, the circuit $\mathsf{C}_i^+$ is contained in the annulus $\text{Ann}(R^{3k_i+2}, R^{3k_i+3})$. For $i = 2, \dots, I$, let $D_i^-(\mathsf{C}_i^-)$ be the event that $\mathsf{C}_i^-$ is the innermost zero-weight circuit around the origin in $\text{Ann}(R^{3k_i},R^{3k_i+1})$, and for $i = 1, \dots, I$, let $D_i^+(\mathsf{C}_i^+)$ be the event that $\mathsf{C}_i^+$ is the outermost zero-weight circuit around the origin in $\text{Ann}(R^{3k_i+2},R^{3k_i+3})$. 

With this notation, the event 
\[
\left( \cap_{i=1}^I \left\{ T_i^- \in A_i^-, T_i^+ \in A_i^+\right\}\right) \cap \{R_n \in A\} \cap  \{\mathcal{I} = I\} \cap  \left( \cap_{i=1}^I \{\mathcal{C}_i^- = \mathsf{C}_i^-, \mathcal{C}_i^+ = \mathsf{C}_i^+\} \right)
\] 
occurs if and only if the following conditions hold:
\begin{enumerate}
\item for all $i = 2, \dots, I$, $T_i^- = T(\mathsf{C}_{i-1}^+, \mathsf{C}_i^-) \in A_i^-$, $D_{i-1}^+(\mathsf{C}_{i-1}^+) \cap D_i^-(\mathsf{C}_i^-)$ occurs, and $E_k^c$ occurs for all $k$ with $k_{i-1}+1 \leq k \leq k_i - 1$,
\item for all $i = 1, \dots, I$, $T_i^+ = T(\mathsf{C}_i^-, \mathsf{C}_i^+) \in A_i^+$, and
\item $R_n = T(\mathsf{C}_I^+, \partial B(n)) \in A$, $D_I^+(\mathsf{C}_I^+)$ occurs, and $E_k^c$ occurs for all $k$ with $k_I+1 \leq k \leq (1/3) \log_R n - 1$.
\end{enumerate}
The events described in (1)-(3) depend on weights for disjoint sets of edges, so they are independent.
Indeed, for each $i$, the event in (1) depends on edge-weights on $\mathsf{C}_{i-1}^+ \cup \mathsf{C}_i^-$ and in the region between these circuits. For each $i$, the event in (2) depends on the edge-weights strictly between the circuits $\mathsf{C}_i^-$ and $\mathsf{C}_i^+$. The event in (3) depends on edge-weights on $\mathsf{C}_I^+$ and in the region between $\mathsf{C}_I^+$ and $\partial B(n)$. By independence, we obtain
\begin{align*}
&\mathbb{P}\left( \cap_{i=1}^I \left\{ T_i^- \in A_i^-, T_i^+ \in A_i^+\right\}, R_n \in A,  \mathcal{I} = I ,    \cap_{i=1}^I \{\mathcal{C}_i^- = \mathsf{C}_i^-, \mathcal{C}_i^+ = \mathsf{C}_i^+\} \right) \\
=~& \prod_{i=2}^I \mathbb{P}\left(T(\mathsf{C}_{i-1}^+,\mathsf{C}_i^-) \in A_i^-, D_{i-1}^+(\mathsf{C}_{i-1}^+), D_i^-(\mathsf{C}_i^-), \cap_{k=k_{i-1}+1}^{k_i-1} E_k^c\right) \\
\times~& \prod_{i=1}^I \mathbb{P}(T(\mathsf{C}_i^-, \mathsf{C}_i^+) \in A_i^+) \\
\times~& \mathbb{P}\left( T(\mathsf{C}_I^+, \partial B(n)) \in A, D_I^+(\mathsf{C}_I^+), \cap_{k=k_I+1}^{\lfloor \frac{1}{3} \log_R n\rfloor -1} E_k^c\right).
\end{align*}
Taking all of $A,A_i^-,A_i^+$ to be equal to $\mathbb{R}$ produces a similar expression for $\mathbb{P}(\mathcal{I}=I, \cap_{i=1}^I \{\mathcal{C}_i^- = \mathsf{C}_i^-, \mathcal{C}_i^+ = \mathsf{C}_i^+\})$. Dividing the two produces the formula
\begin{align}
&\mathbb{P}\left( \cap_{i=1}^I \left\{ T_i^- \in A_i^-, T_i^+ \in A_i^+\right\}, R_n \in A \mid \mathcal{I} = I, \cap_{i=1}^I \{\mathcal{C}_i^- = \mathsf{C}_i^-, \mathcal{C}_i^+ = \mathsf{C}_i^+\} \right)  \nonumber \\
=~&\prod_{i=2}^I \mathbb{P}\left(T(\mathsf{C}_{i-1}^+, \mathsf{C}_i^-) \in A_i^- \mid D_{i-1}^+(\mathsf{C}_{i-1}^+), D_i^-(\mathsf{C}_i^-), \cap_{k=k_{i-1}+1}^{k_i-1}E_k^c\right)  \nonumber \\
\times~&\prod_{i=1}^I \mathbb{P}(T(\mathsf{C}_i^-,\mathsf{C}_i^+) \in A_i^+)  \label{eq: formula_man}\\\times~&\mathbb{P}\left(T(\mathsf{C}_I^+,\partial B(n)) \in A \mid D_I^+(\mathsf{C}_I^+), \cap_{k=k_I+1}^{\lfloor \frac{1}{3} \log_R n\rfloor-1} E_k^c\right). \nonumber
\end{align}
Last, taking all Borel sets except for one (say $A_1^+$) equal to $\mathbb{R}$ shows that this expression equals the right side of \eqref{eq: independence} and that the factor in \eqref{eq: formula_man} is equal to the left side of \eqref{eq: free_percolation}.
\end{proof}

\subsection{Step 3: application of Lem.~\ref{lem: resampling}.}\label{sec: resampling_section}

In this step, we use the circuit construction of step 2, along with the Lem.~\ref{lem: resampling}, to prove the reduced claim \eqref{eq: simpler_condition}, and this will complete the proof of Thm.~\ref{thm: main_result}. As before, we fix $L \geq 0$, an integer $K \geq 0$, and $\eta>1/2$.

First we need to ensure that sufficiently many events $E_k$ occur inside the box $B(n)$. Let $\epsilon>0$. By \eqref{eq: E_k_lower_bound}, we can fix $R$ so large that $\mathbb{P}(E_k) \geq 1-\epsilon$ for all $k$. Next, for $m \geq K$, let $G_m$ be the event
\[
G_m = \left\{\sum_{k=K}^{m'} \mathbf{1}_{E_k} \geq \frac{m'}{4} \text{ for all } m' \geq m\right\}.
\]
Because the events $E_k$ depend on weights in disjoint annuli, they are independent. The variables $\mathbf{1}_{E_k}$ stochastically dominate a family of i.i.d.~Bernoulli random variables with parameter $1-\epsilon$. Therefore a.s., $\liminf_{m \to \infty}  (1/m)\sum_{k=1}^m \mathbf{1}_{E_k} \geq 1/2$, and so, given $\epsilon>0$, we can fix $m \geq K$ such that $\mathbb{P}(G_m) \geq 1-\epsilon$. Because $E_k$ is a decreasing event of the edge-weights (and therefore so is $G_m$), as is $\{T(0,\partial B(n)) \leq L\}$, the FKG inequality implies that
\begin{align*}
\mathbb{P}(E_K^c \cup G_m^c \mid T(0,\partial B(n)) \leq L) &\leq \mathbb{P}(E_K^c \mid T(0,\partial B(n)) \leq L) + \mathbb{P}(G_m^c \mid T(0,\partial B(n)) \leq L) \\
&\leq \mathbb{P}(E_K^c) + \mathbb{P}(G_m^c) \\
&\leq 2\epsilon.
\end{align*}
In what follows, we will think of $n$ as much larger than our fixed $K$ and $m$.

Proceeding to the left side of \eqref{eq: simpler_condition},
\begin{align}
&\mathbb{P}(T(0,\partial B(K)) \geq F^{-1}(\eta) \mid T(0,\partial B(n)) \leq L) \nonumber \\
\leq~& 2\epsilon +\mathbb{P}(T(0,\partial B(K)) \geq F^{-1}(\eta), E_K,G_m \mid T(0,\partial B(n)) \leq L). \label{eq: caprese_salad}
\end{align}
On the event $\{T(0,\partial B(K)) \geq F^{-1}(\eta), E_K, G_m\}$, we will condition on the values of the circuits $\mathcal{C}_i$ from the circuit decomposition. For this reason, we need to define our class of relevant collections of circuits. For our $R,L>0$, and our integers $m,n \geq 1$, let $\Xi = \Xi(m,n,R,L)$ be the collection of pairs $(I, \{\mathsf{C}_i^\pm\}_{i=1}^I)$, where $I \geq 1$ is an integer, and $\{\mathsf{C}_i^\pm\}_{i=1}^I$ are circuits satisfying the following conditions:
\begin{enumerate}
\item $\mathsf{C}_1^- = \{0\}$ and $\mathsf{C}_1^+ \subset \text{Ann}(R^{3K+2},R^{3K+3})$,
\item for $i =2, \dots, I$, $\mathsf{C}_i^-$ is contained in an annulus $\text{Ann}(R^{3k_i},R^{3k_i+1})$ and $\mathsf{C}_i^+$ is contained in $\text{Ann}(R^{3k_i+2},R^{3k_i+3})$, with $K < k_2 < k_3 < \dots < k_I$ and $B(R^{3k_I+3}) \subset B(n)$,
\item for all $m' \geq m$ such that $B(R^{3m'+3})$ is contained in $B(n)$, we have
\[
\sum_{k=K}^{m'} \mathbf{1}_{\{k = k_i \text{ for some }i = 1, \dots, I\}} \geq \frac{m'}{4},
\]
\item and $\mathbb{P}(\mathcal{I}=I, \cap_{i=1}^I \{\mathcal{C}_i^- = \mathsf{C}_i^-, \mathcal{C}_i^+ = \mathsf{C}_i^+\} \mid T(0,\partial B(n)) \leq L) > 0$.
\end{enumerate}
Condition 3 is similar to the definition of $G_m$, and all four conditions ensure that if $E_K \cap G_m$ occurs, then the (random) pair $(\mathcal{I}, \{\mathcal{C}_i^\pm\}_{i=1}^{\mathcal{I}})$ lies in $\Xi$. With this definition, the term in \eqref{eq: caprese_salad} is
\begin{align}
&\mathbb{P}(T(0,\partial B(K)) \geq F^{-1}(\eta), E_K,G_m \mid T(0,\partial B(n)) \leq L) \nonumber \\
\leq~& \mathbb{P}(T(0,\partial B(K)) \geq F^{-1}(\eta), (\mathcal{I}, \{\mathcal{C}_i^\pm\}) \in \Xi \mid T(0,\partial B(n)) \leq L) \nonumber \\
=~&\sum_{(I,\{\mathsf{C}_i^\pm\}) \in \Xi} \bigg[ \mathbb{P}(T(0,\partial B(K)) \geq F^{-1}(\eta) \mid T(0,\partial B(n)) \leq L, \mathcal{I}=I, \cap_{i=1}^I \{\mathcal{C}_i^- = \mathsf{C}_i^-, \mathcal{C}_i^+ = \mathsf{C}_i^+\}) \label{eq: pasta_alfredo} \\
&\hspace{1in} \times  \mathbb{P}\left(\mathcal{I}=I, \cap_{i=1}^I \{\mathcal{C}_i^- = \mathsf{C}_i^-, \mathcal{C}_i^+ = \mathsf{C}_i^+\} \mid T(0,\partial B(n)) \leq L \right) \bigg] \nonumber
\end{align}

To estimate \eqref{eq: pasta_alfredo}, we use the decomposition in \eqref{eq: T_decomposition} and observe that $T(0,\partial B(K)) \geq F^{-1}(\eta)$ implies that $T_1^+ \geq F^{-1}(\eta)$. Writing $\mathbb{P}^\ast$ for the conditional measure $\mathbb{P}(\cdot \mid \mathcal{I}=I, \cap_{i=1}^I \{\mathcal{C}_i^- = \mathsf{C}_i^-, \mathcal{C}_i^+ = \mathsf{C}_i^+\})$, the probability in \eqref{eq: pasta_alfredo} becomes 
\begin{align}
&\mathbb{P}^\ast(T(0,\partial B(K)) \geq F^{-1}(\eta) \mid T(0,\partial B(n)) \leq L) \nonumber \\
\leq~& \mathbb{P}^\ast(T_1^+ \geq F^{-1}(\eta) \mid T_1^++T_2^- + T_2^+ + \dots + T_I^- + T_I^+ + R_n \leq L). \label{eq: about_to_apply}
\end{align}
By Lem.~\ref{lem: independence}, the nonnegative variables $T_i^\pm$ and $R_n$ are all independent under $\mathbb{P}^\ast$. Therefore we can apply Lem.~\ref{lem: resampling} with $X_1 = T_1^+$ and $\delta = \delta' = F^{-1}(\eta)$ to obtain
\[
\eqref{eq: about_to_apply} \leq \frac{L}{\mathbb{P}^\ast(T_1^+ =0) \sum_{i=2}^I \mathbb{E}^\ast T_i^+ \mathbf{1}_{\{T_i^+ \leq F^{-1}(\eta)\}}}.
\]
Here, $\mathbb{E}^\ast$ represents expectation relative to $\mathbb{P}^\ast$, and we have removed the terms $\mathbb{E}^\ast T_i^- \mathbf{1}_{\{T_i^- \leq F^{-1}(\eta)\}}$ and $\mathbb{E}^\ast R_n \mathbf{1}_{\{T_n \leq F^{-1}(\eta)\}}$ from the denominator in Lem.~\ref{lem: resampling} because they are nonnegative. From \eqref{eq: free_percolation}, the distribution of $T_i^+$ under $\mathbb{P}^\ast$ is the same as that of $T(\mathsf{C}_i^-,\mathsf{C}_i^+)$ under $\mathbb{P}$, so the above becomes
\begin{align}
\eqref{eq: about_to_apply} &\leq \frac{L}{\mathbb{P}(T(0,\mathsf{C}_1^+) =0) \sum_{i=2}^I \mathbb{E} T(\mathsf{C}_i^-,\mathsf{C}_i^+) \mathbf{1}_{\{T(\mathsf{C}_i^-,\mathsf{C}_i^+) \leq F^{-1}(\eta)\}}} \nonumber \\
&\leq \frac{2^{R^{3K+3}} L}{\sum_{i=2}^I \mathbb{E} T(\mathsf{C}_i^-,\mathsf{C}_i^+) \mathbf{1}_{\{T(\mathsf{C}_i^-,\mathsf{C}_i^+) \leq F^{-1}(\eta)\}}}. \label{eq: final_stretch_maybe}
\end{align}

To complete the proof, we must estimate the summands in the denominator of \eqref{eq: final_stretch_maybe}, and for this, we need some constructions from Bernoulli percolation. In the following lemma, if $1/2 + 2^{-\Cl[smc]{c: construction_constant_1}k} \geq 1$, we interpret $F^{-1}(1/2 + 2^{-\Cr{c: construction_constant_1}k})$ to be $+\infty$.
\begin{Lemma}\label{lem: percolation_lemma}
There exist constants $\Cl[lgc]{c: construction_constant_2},\Cr{c: construction_constant_1}>0$ such that for any $k \geq 1$, and any deterministic pair of circuits $\mathfrak{C}_1, \mathfrak{C}_2$ around the origin in $\mathrm{Ann}(R^{3k},R^{3k+1})$ and $\mathrm{Ann}(R^{3k+2},R^{3k+3})$ respectively,
\[
\mathbb{P}\left( F^{-1}\left( \frac{1}{2} + \frac{1}{2^{\Cr{c: construction_constant_2}k}}\right) \leq T(\mathfrak{C}_1,\mathfrak{C}_2) \leq F^{-1}\left( \frac{1}{2} + \frac{1}{2^{\Cr{c: construction_constant_1}k}}\right) \right) \geq \Cr{c: construction_constant_1}.
\]
\end{Lemma}

We first show how to complete the proof of Thm.~\ref{thm: main_result} given Lem.~\ref{lem: percolation_lemma}. Then we give the proof of the lemma. Applying it to the circuits $\mathsf{C}_i^\pm$ from \eqref{eq: final_stretch_maybe}, we obtain
\[
\mathbb{E}T(\mathsf{C}_i^-,\mathsf{C}_i^+) \mathbf{1}_{\{T(\mathsf{C}_i^-,\mathsf{C}_i^+) \leq F^{-1}(\eta)\}} \geq \Cr{c: construction_constant_1} F^{-1}\left( \frac{1}{2} + \frac{1}{2^{\Cr{c: construction_constant_2}k_i}}\right)
\]
so long as $1/2 + 2^{-\Cr{c: construction_constant_1}k_i} \leq \eta$. Returning to \eqref{eq: final_stretch_maybe}, we find
\[
\eqref{eq: about_to_apply} \leq \frac{2^{R^{3K+3}} L}{\Cr{c: construction_constant_1} \widetilde{\sum}_{i=2, \dots, I} F^{-1}\left( \frac{1}{2} + \frac{1}{2^{\Cr{c: construction_constant_2}k_i}}\right)}.
\]
Here, $\widetilde{\sum}$ is the summation restricted to the $i$ such that $2^{-\Cr{c: construction_constant_1}k_i} \leq \eta - 1/2$. Putting this back in \eqref{eq: pasta_alfredo} and then back in \eqref{eq: caprese_salad} produces
\begin{equation}\label{eq: to_slide_in_later}
\mathbb{P}(T(0,\partial B(K)) \geq F^{-1}(\eta) \mid T(0,\partial B(n)) \leq L) \leq 2 \epsilon + \max_{(I,\{\mathsf{C}_i^\pm\}) \in \Xi} \left[ \frac{2^{R^{3K+3}} L}{\Cr{c: construction_constant_1} \widetilde{\sum}_{i=2, \dots, I} F^{-1}\left( \frac{1}{2} + \frac{1}{2^{\Cr{c: construction_constant_2}k_i}}\right)} \right].
\end{equation}
The sum in the denominator equals
\begin{align*}
&\sum_{k=K}^\infty \mathbf{1}_{\{k = k_i \text{ for some }i=2, \dots, I\}} \mathbf{1}_{\{2^{-\Cr{c: construction_constant_1}k} \leq \eta - \frac{1}{2}\}} F^{-1}\left( \frac{1}{2} + \frac{1}{2^{\Cr{c: construction_constant_2}k}}\right) \\
\geq~&\sum_{k=K}^\infty \mathbf{1}_{\{k = k_i \text{ for some }i=2, \dots, I\}} \mathbf{1}_{\{2^{-\Cr{c: construction_constant_1}k} \leq \eta - \frac{1}{2}\}} a_{\lceil \Cr{c: construction_constant_2}k \rceil}.
\end{align*}
Here, we have brought back the notation $a_k = F^{-1}(1/2 + 1/2^k)$ and used monotonicity. For our fixed $\eta$, we can find $\Cl[lgc]{c: finish_it_up}>0$ such that the above is not smaller than
\begin{equation}\label{eq: near_end_lower_bound}
- \Cr{c: finish_it_up} + \sum_{k=K}^{\lfloor \frac{1}{3} \log_R n \rfloor -1} \mathbf{1}_{\{k = k_i \text{ for some }i=1, \dots, I\}} a_{\lceil \Cr{c: construction_constant_2}k \rceil}.
\end{equation}

To estimate the sum in \eqref{eq: near_end_lower_bound}, we use the following elementary lemma from \cite[Lem.~2.1]{DHHL21}. Suppose that $(\mathsf{a}_k), (\mathsf{b}_k),$ and $(\mathsf{c}_k)$ are nonnegative sequences such that $(\mathsf{b}_k)$ is nonincreasing. Writing $\mathsf{A}_j = \sum_{k=1}^j \mathsf{a}_k$ for $j \geq 1$ (and similarly for $\mathsf{C}_j$), suppose that $\mathsf{A}_j \leq \mathsf{C}_j$ for all $j \geq 1$. Then
\begin{equation}\label{eq: sum_by_parts}
\sum_{k=1}^j \mathsf{a}_k \mathsf{b}_k \leq \sum_{k=1}^j \mathsf{c}_k \mathsf{b}_k \text{ for all } j \geq 1.
\end{equation}
We put $\mathsf{c}_k = \mathsf{1}_{\{k = k_i \text{ for some } i = 1, \dots, I\}}$, $\mathsf{b}_k = a_{\lceil C_{200}k\rceil}$, and
\[
\mathsf{a}_k = \frac{1}{4} \mathbf{1}_{\{m \leq k \leq \lfloor \frac{1}{3} \log_R n\rfloor - 1\}},
\]
and apply the shifted version of \eqref{eq: sum_by_parts}, for sums starting from $k=K$ instead of from $k=1$. Because the pair $(I, \{\mathsf{C}_i^\pm\})$ is in $\Xi$, item (3) of the definition of $\Xi$ gives that $\mathsf{A}_j \leq \mathsf{C}_j$ for all $j$, and so the expression in \eqref{eq: near_end_lower_bound} is at least equal to
\[
-\Cr{c: finish_it_up} + \sum_{k=K}^{\lfloor \frac{1}{3} \log_R n \rfloor-1} \frac{1}{4} \mathbf{1}_{\{k \geq m\}} a_{\lceil \Cr{c: construction_constant_2}k\rceil}.
\]
Returning to \eqref{eq: to_slide_in_later},
\[
\mathbb{P}(T(0,\partial B(K)) \geq F^{-1}(\eta) \mid T(0,\partial B(n)) \leq L) \leq 2\epsilon + \frac{2^{R^{3K+3}}L}{\Cr{c: construction_constant_1}\left( -\Cr{c: finish_it_up} +  \frac{1}{4}\sum_{k=m}^{\lfloor \frac{1}{3} \log_R n \rfloor - 1} a_{\lceil \Cr{c: construction_constant_2}k\rceil}\right)}.
\]
We have assumed that $\sum_{k=2}^\infty a_k = \infty$, so as $n \to \infty$, the sum in the denominator diverges. Therefore
\[
\limsup_{n \to \infty} \mathbb{P}(T(0,\partial B(K)) \geq F^{-1}(\eta) \mid T(0,\partial B(n)) \leq L) \leq 2 \epsilon.
\]
Since $\epsilon$ is arbitrary, this shows \eqref{eq: simpler_condition} and completes the proof of Thm.~\ref{thm: main_result}, assuming that we also prove Lem.~\ref{lem: percolation_lemma}.

\begin{proof}[Proof of Lem.~\ref{lem: percolation_lemma}]
The argument for Lem.~\ref{lem: percolation_lemma} uses percolation constructions that have appeared in various places and are now standard. For this reason, we omit the details of the main estimate and only outline the idea.

To describe the strategy, we will need a few definitions. First, we represent our weights $t_e$ as images of uniform variables. Let $(U_e)_{e \in \mathbb{E}^2}$ be a collection of i.i.d.~Uniform$(0,1)$ random variables assigned to the edges, and (re)define $t_e = F^{-1}(U_e)$. Then the $t_e$ are i.i.d.~and have common distribution function $F$. For $p \in (0,1)$, say that an edge $e$ is $p$-open if $U_e \leq p$; otherwise, $e$ is $p$-closed. A central tool in the study of critical percolation is correlation length, and we take the definition from \cite[Eq.~1.21]{kestenscaling}. For $n_0 \geq 1$, and $p \in (1/2,1)$, let
\[
\sigma(n_0,p) = \mathbb{P}(\text{there is a }p\text{-open left-right crossing of } B(n_0)).
\]
The term ``$p$-open left-right crossing'' refers to a path in $B(n_0)$ whose initial vertex is in the left side $\{-n_0\} \times [-n_0,n_0]$, whose final vertex is in the right side $\{n_0\} \times [-n_0,n_0]$, and all of whose edges are $p$-open. It is known that for $p \in (1/2,1)$, $\sigma(n_0,p) \to 1$ as $n_0 \to \infty$. Accordingly, for $\epsilon>0$ and $p \in (1/2,1)$, we set
\[
L(p,\epsilon) = \min\{n_0 : \sigma(n_0,p) \geq 1-\epsilon\}.
\]
This $L(p,\epsilon)$ is called the (finite-size scaling) correlation length. It is known that $\lim_{p \downarrow 1/2} L(p,\epsilon) = \infty$ for $\epsilon>0$, and that there exists $\epsilon_1>0$ such that for all $\epsilon,\epsilon' \in (0,\epsilon_1)$, the ratio $L(p,\epsilon)/L(p,\epsilon')$ is bounded away from $0$ and $\infty$ as $p \downarrow 1/2$. We thus define $L(p) = L(p,\epsilon_1)$ for simplicity. For $k \geq 1$, define
\[
p_k = \min\{p \in (1/2,1) : L(p) \leq R^{3k}\}.
\]
As explained in \cite[Eq.~(2.5)]{DLW17}, there exist constants $\Cl[smc]{c: percolation_1},\Cl[lgc]{c: percolation_2}>0$ such that for all $k$,
\begin{equation}\label{eq: p_k_bounds}
\frac{1}{2} + R^{-\Cr{c: percolation_2}k} \leq p_k \leq \frac{1}{2} + R^{-\Cr{c: percolation_1}k}.
\end{equation}

In addition to the correlation length, we need to introduce the notion of duality. The dual lattice has vertex set $(\mathbb{Z}^2)^\ast = \left\{v + \left( \frac{1}{2}, \frac{1}{2}\right) : v \in \mathbb{Z}^2\right\}$, with edge set $(\mathbb{E}^2)^\ast$, consisting of those edges whose endpoints are Euclidean distance 1 from each other. A dual edge $e^\ast$ bisects a unique primal edge $e$ (edge from the original lattice). For $p \in (0,1)$, we say that a dual edge is $p$-open if its corresponding primal edge is $p$-open; otherwise, it is $p$-closed.

Given our deterministic circuits, $\mathfrak{C}_1,\mathfrak{C}_2$, we define an event with (uniformly) positive probability that guarantees that the passage time between the circuits is in the interval $(F^{-1}(1/2 + 2^{-\Cr{c: construction_constant_2}k}), F^{-1}(1/2 + 2^{-\Cr{c: construction_constant_1}k}))$ for suitable constants $\Cr{c: construction_constant_2},\Cr{c: construction_constant_1}$. Let $\mathfrak{O}_k$ be the event that there exists an edge $e$ in $\text{Ann}(R^{3k+1},R^{3k+2})$ such that all of the following hold:
\begin{enumerate}
\item $U_e \in (p_k, 2p_k - 1/2)$,
\item one endpoint of $e$ is connected by a $1/2$-open path to $\partial B(R^{3k})$ and the other is connected by an edge-disjoint $(1/2)$-open path to $\partial B(R^{3k+3})$, and
\item both endpoints of the dual edge $e^\ast$ are connected by a $(2p_k-1/2)$-closed dual path in $\mathrm{Ann}(R^{3k}, R^{3k+3})$ which surrounds the origin. That is, the union of this dual path with $e^\ast$ forms a dual circuit around the origin.
\end{enumerate}

\begin{figure}
		\centering
		\includegraphics[width=0.8\linewidth, trim={4cm 0 0 10cm}, clip]{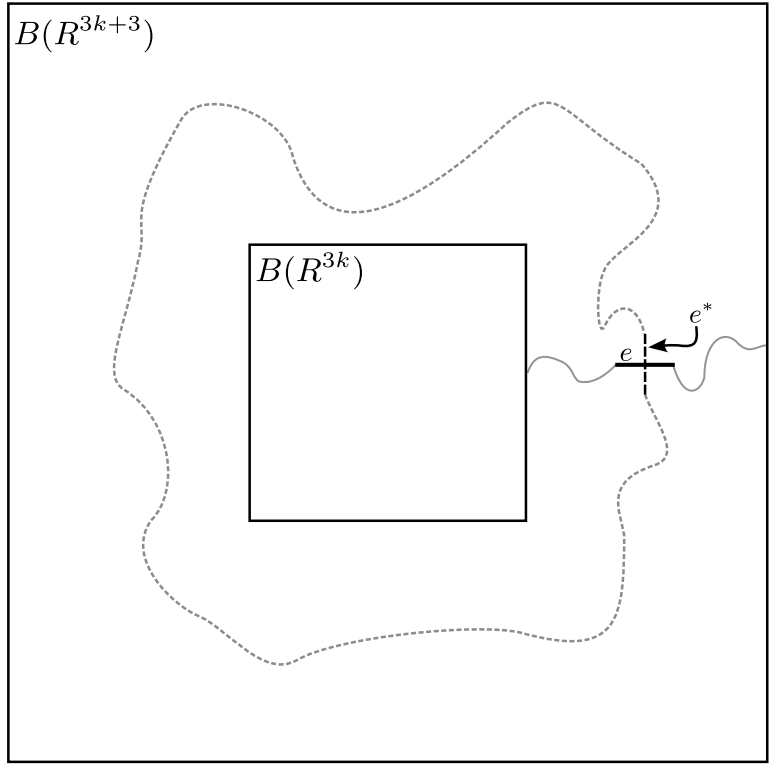}
		\caption{Depiction of the event $\mathfrak{D}_k$. The edge $e$, along with its dual edge, $e^\ast$, appears in the right side of the annulus $\mathrm{Ann}(R^{3k}, R^{3k+3})$. The two paths decribed in item 2 of the definition of $\mathfrak{D}_k$ are drawn solid, and the path from item 3 is drawn dotted.}
		\label{fig: construction}
	\end{figure}

(See Fig.~\ref{fig: construction}.) One can show that there exists $\Cl[smc]{c: percolation_3}>0$ such that for any $k \geq 1$, one has 
\begin{equation}\label{eq: gluing_estimate}
\mathbb{P}(\mathfrak{O}_k) \geq \Cr{c: percolation_3}.
\end{equation} 
This comes from a ``gluing'' argument using the Russo-Seymour-Welsh theorem and the generalized FKG inequality; see \cite[Prop.~2.3]{DH21} or the proof of Thm.~1.5 in \cite{DS11} for very similar proofs. The basic idea is as follows. For a given edge $e$, the probability that item (1) holds is $p_k-1/2$. The probability that items (2) and (3) hold (using the generalized FKG inequality and the RSW theorem) is of the same order as the ``alternating four-arm probability,'' written $\pi_4(R^{3k})$. This is the probability that the endpoints of $e$ are connected by two disjoint $1/2$-open paths to the boundary of a box of radius $R^{3k}$ centered at the midpoint of $e$, and the endpoints of $e^\ast$ are connected by two disjoint $1/2$-closed paths to the same boundary. In this estimate, one must use \cite[Lem.~6.3]{DSV09} to change the $(2p_k-1/2)$-closed dual paths from item (3) to $1/2$-closed ones. Because there can be at most one edge in $\text{Ann}(R^{3k+1},R^{3k+2})$ satisfying items (1)-(3), we get
\[
\mathbb{P}(\mathfrak{O}_k) = \sum_{e \in \text{Ann}(R^{3k+1},R^{3k+2})} \mathbb{P}(\text{(1)-(3) hold at }e),
\]
and the right side is comparable to $(R^{3k})^2 (p_k-1/2) \pi_4(R^{3k})$. By using the scaling relation $(R^{3k})^2 (p_k-1/2) \pi_4(R^{3k}) \geq \Cr{c: percolation_3}$ from \cite[Prop.~34]{kestenscaling}, we obtain \eqref{eq: gluing_estimate}.

Given the lower bound in \eqref{eq: gluing_estimate}, we are left to estimate the passage time $T(\mathfrak{C}_1,\mathfrak{C}_2)$ on the event $\mathfrak{O}_k$. So suppose that $\mathfrak{O}_k$ occurs. Form a path $\gamma_1$ connecting $\mathfrak{C}_1$ to $\mathfrak{C}_2$ by following one $(1/2)$-open path from item (2) from $\mathfrak{C}_1$ to the edge $e$, traversing $e$, then following the other $(1/2)$-open path from item (2) from $e$ to $\mathfrak{C}_2$. Each $(1/2)$-open edge has weight $\leq F^{-1}(1/2) =0$, so this path has total passage time $t_e \leq F^{-1}(2p_k-1/2)$. By \eqref{eq: p_k_bounds}, we have
\[
T(\mathfrak{C}_1, \mathfrak{C}_2) \leq T(\gamma_1) \leq F^{-1}\left( \frac{1}{2} + 2\left(p_k - \frac{1}{2}\right)\right) \leq F^{-1}\left( \frac{1}{2} + 2R^{-\Cr{c: percolation_1}k}\right).
\]
On the other hand, if $\gamma_2$ is any path connecting $\mathfrak{C}_1$ to $\mathfrak{C}_2$, by duality, it must contain either $e$ or an edge whose dual is in the dual path described in item (3). Again by \eqref{eq: p_k_bounds},
\[
T(\mathfrak{C}_1, \mathfrak{C}_2) = \inf_{\gamma_2 : \mathfrak{C}_1 \to \mathfrak{C}_2} T(\gamma_2) \geq F^{-1}(p_k) \geq F^{-1} \left( \frac{1}{2} + R^{-\Cr{c: percolation_2}k}\right).
\]
Combining these two inequalities, and adjusting the constants shows that for suitable $\Cl[lgc]{c: percolation_7},\Cl[smc]{c: percolation_8}>0$, on the event $\mathfrak{O}_k$, one has
\[
F^{-1}\left( \frac{1}{2} + \frac{1}{2^{\Cr{c: percolation_7}k}}\right) \leq T(\mathfrak{C}_1,\mathfrak{C}_2) \leq F^{-1}\left( \frac{1}{2} + \frac{1}{2^{\Cr{c: percolation_8}k}}\right).
\]
This completes the proof.
\end{proof}


\section{Lemma on nonnegative sums}\label{sec: nonnegative_sums}

In this section, we prove Lem.~\ref{lem: resampling}, which was used in the proof of the main result, Thm.~\ref{thm: main_result}. Given an independent sequence $(X_n)$ of nonnegative random variables, the lemma gives an upper bound on the probability that $X_1$ is positive given that the sum $X_1 + \dots + X_n$ is no bigger than a constant $L$. We are interested in the case when the sum diverges a.s., so the event in question is a type of extreme deviation event. 

In addition to being crucial for our main theorem, the lemma suggests an interesting problem.

\begin{Question}\label{question: sequence_question}
For which sequences of nonnegative independent random variables $(X_n)$ is the conditional distribution of $X_1$ (or the joint distribution of any finite number of variables) given that $X_1 + \dots + X_n \leq L$ trivial in the limit as $n \to \infty$? 
\end{Question}
We do not have a complete answer to this question. As we will discuss after the proof of Lem.~\ref{lem: resampling}, our inequality implies that the answer is yes (that is, the limiting distribution is trivial) for a large class of $(X_n)$: those for which \eqref{eq: sum_diverges_condition} holds. We will also give a few examples that show that the answer can be no for certain sequences $(X_n)$. These results suggest that if there is a necessary and sufficient condition for the distribution to be trivial, it might not have a concise form.

\begin{Lemma}\label{lem: resampling}
Let $X_1, X_2, \dots, X_n$ be independent, nonnegative random variables and set $S_n = X_1 + \dots + X_n$ for $n \geq 1$. For $L \geq 0$ such that $\mathbb{P}(S_n \leq L) > 0$ and any $\delta,\delta'$ with $0 < \delta' \leq \delta$, we have
\[
\mathbb{P}(X_1 \geq \delta \mid S_n \leq L) \leq \frac{L}{\mathbb{P}(X_1\leq \delta-\delta' ) \sum_{k=2}^n \mathbb{E}X_k\mathbf{1}_{\{X_k \leq \delta'\}}}.
\]
\end{Lemma}
\begin{proof}
We use a resampling argument. For $k \geq 2$, let $X_1',X_k'$ be independent random variables, independent of $X_1, \dots, X_n$, such that $X_1,X_1'$ have the same distribution, as do $X_k,X_k'$. Write
\[
S_n' = S_n - X_1 - X_k + X_1'+X_k'
\]
as the sum with $X_1,X_k$ replaced by $X_1',X_k'$. The conditions
\begin{enumerate}
\item $S_n' \leq L$ and $X_1' \geq \delta$, and
\item $X_1 \leq \delta - \delta'$ and $X_k \leq \delta'$
\end{enumerate}
together imply that
\[
S_n = S_n' - X_1'-X_k'+X_1+X_k \leq L-\delta -0+\delta-\delta' + \delta' = L.
\]
Combining this with independence yields for any $\delta'' \in [0,\delta']$
\[
\mathbb{P}(S_n \leq L, X_k \in [\delta'',\delta']) \geq \mathbb{P}(S_n' \leq L, X_1' \geq \delta) \mathbb{P}(X_1\leq \delta - \delta') \mathbb{P}(X_k \in [\delta'',\delta'])
\]
or, after dividing by $\mathbb{P}(S_n \leq L) = \mathbb{P}(S_n'\leq L)$,
\begin{align*}
\mathbb{P}(X_k \in [\delta'',\delta'] \mid S_n \leq L) &\geq \mathbb{P}(X_1' \geq \delta \mid S_n' \leq L) \mathbb{P}(X_1\leq \delta - \delta') \mathbb{P}(X_k \in [\delta'',\delta']) \\
&= \mathbb{P}(X_1 \geq \delta \mid S_n \leq L) \mathbb{P}(X_1\leq \delta - \delta') \mathbb{P}(X_k \in [\delta'',\delta']).
\end{align*}
If we integrate over $\delta'' \in [0,\delta']$ and sum from $k=2$ to $k=n$, we obtain
\[
L \geq \sum_{k=2}^n \mathbb{E}\left( X_k \mathbf{1}_{\{X_k \leq \delta'\}} \mid S_n \leq L\right) \geq \mathbb{P}(X_1 \geq \delta \mid S_n \leq L) \mathbb{P}(X_1\leq \delta - \delta') \sum_{k=2}^n \mathbb{E}X_k \mathbf{1}_{\{X_k \leq \delta'\}},
\]
which implies the result.
\end{proof}

\appendix

\section{Random series conditioned to be of constant order}
As promised earlier, we now provide examples of different limiting behavior related to Question~\ref{question: sequence_question}. We begin with the trivial limiting behavior that falls into the scope of Lem.~\ref{lem: resampling}. Let $X_1, X_2, \dots$ be an infinite sequence of independent, nonnegative random variables and for $k \geq 1$, let $I_k = \inf\{x \geq 0 : \mathbb{P}(X_k \leq x) > 0\}$ be the infimum of the support of the distribution of $X_k$. Suppose that
\begin{equation}\label{eq: sum_diverges_condition}
\text{for any }\delta'>0, \hspace{.25in} \sum_{k=1}^\infty \mathbb{E} X_k \mathbf{1}_{\{X_k \leq \delta'\}} = \infty \text{ and } \sum_{k=1}^\infty I_k <\infty. \hspace{.5in}
\end{equation}
(Of course, if $\sum_{k=1}^\infty I_k = \infty$, then the event $\{S_n \leq L\}$ is eventually empty.) Then, for any $L >  \sum_{k=1}^\infty I_k$ and $i \geq 1$, the conditional distribution of $(X_1, \dots, X_i)$ given the event $\{S_n \leq L\}$ converges to the (trivial) product measure $\prod_{j=1}^i \delta_{\{I_j\}}$ as $n \to \infty$. To see why, let $Z_1 = X_1 + \dots + X_i$, and $Z_k = X_{i+k-1}$ for $k \geq 2$. For any $\delta > 0$, Lemma~\ref{lem: resampling} applied to $(Z_k)$ gives
\[
\mathbb{P}(Z_1 \geq I_1 + \dots + I_i + \delta \mid S_n \leq L) \leq \frac{L}{\mathbb{P}\left(Z_1 \leq I_1 + \dots + I_i + \frac{\delta}{2}\right) \sum_{k=2}^n \mathbb{E}Z_k \mathbf{1}_{\left\{Z_k \leq \frac{\delta}{2}\right\}}},
\]
and the right side converges to $0$ as $n \to \infty$. Thus, when \eqref{eq: sum_diverges_condition} holds, the limiting distribution of $(X_n)$ is trivial. 

\begin{Rem}\label{rem: gap_remark}
It is possible to weaken the condition in \eqref{eq: sum_diverges_condition} and still get a trivial limit. For example, if there exists $\epsilon>0$ such that $\mathbb{P}(X_k \leq I_k) = \mathbb{P}(X_k \leq I_k + \epsilon)$ for all $k$---there is a ``gap'' of size $\epsilon$ in the support of the distributions---then \eqref{eq: sum_diverges_condition} need only hold for $\delta' = \epsilon$.
\end{Rem}

For the remainder of the section, we give examples that do not fall into the setting of Lem.~\ref{lem: resampling}.

\subsection{Parity examples}

First we present examples whose limiting distribution depends on $L$. Fix any $p \in (0,1)$ and let $(X_n)$ be an independent sequence satisfying
\[
\mathbb{P}(X_1 = 0) = p = 1-\mathbb{P}(X_1=1)
\]
and
\[
\mathbb{P}(X_n = 0) = p = 1-\mathbb{P}(X_n=2) \text{ for } n \geq 2.
\]
If $L$ is an positive odd integer, then for $n \geq 2$,
\[
\{S_n \leq L\} = \{X_2 + \dots + X_n \leq L\} = \{X_2 + \dots + X_n \leq L-1\},
\]
so
\begin{equation}\label{eq: parity_claim_1}
\mathbb{P}(X_1=1 \mid S_n \leq L) = \mathbb{P}(X_1=1 \mid X_2 + \dots + X_n \leq L-1) = \mathbb{P}(X_1=1) = 1-p.
\end{equation}

On the other hand, we will show that if $L$ is a positive even integer, then
\begin{equation}\label{eq: parity_claim}
\lim_{n \to \infty} \mathbb{P}(X_1 = 1 \mid S_n \leq L) = 0.
\end{equation}
To prove this, we observe if $L = 2k$, then $\mathbb{P}(X_1=1 \mid S_n \leq L)$ equals
\begin{align*}
& \dfrac{\mathbb{P}(X_1 = 1) \mathbb{P}( X_2 + \dots + X_n \leq L - 1  )}{\mathbb{P}(X_1 = 1) \mathbb{P}( X_2 + \dots + X_n \leq L - 1  ) +\mathbb{P}(X_1 = 0)  \mathbb{P}( X_2 + \dots + X_n \leq L   ) } \\
=~& \dfrac{\mathbb{P}(X_1 = 1) \mathbb{P}( X_2 + \dots + X_n \leq L - 2 )}{\mathbb{P}(X_1 = 1) \mathbb{P}( X_2 + \dots + X_n \leq L - 2  ) +\mathbb{P}(X_1 = 0)  \mathbb{P}( X_2 + \dots + X_n \leq L   ) } \\
  \leq~& \frac{1-p}{p} \cdot \dfrac{\mathbb{P}( X_2 + \dots + X_n \leq L-2 )}{ \mathbb{P}( X_2 + \dots + X_n = L  ) } \\
=~& \frac{1-p}{p} \cdot \dfrac{\sum_{i=0}^{k-1} { n \choose i } (1-p)^i p^{n-i} }{    { n \choose k } (1-p)^k p^{n-k}   } \to 0 \text{ as } n \to \infty.
\end{align*} 
	
%
	

\subsection{General parity result}

Equations \eqref{eq: parity_claim_1} and \eqref{eq: parity_claim} show that in the parity example, the conditional distribution of $X_1$ tends to a different limit as $n \to \infty$ depending on whether $L$ is even or odd. One can modify this example in various ways. The general form that these modifications take will be analyzed in this section.

We now suppose that $X_1, X_2, \dots$ are independent nonnegative random variables such that
\[
X_2, X_3, \dots \text{ are i.i.d. with } p: = \mathbb{P}(X_k > 0) \in (0,1] \text{ for } k \geq 2.
\]
(If $p =0$ then $S_n = X_1$ for all $n \geq 2$ and the question becomes trivial.) To state the general result, we define random variables $\mathfrak{X}_2, \mathfrak{X}_3, \dots$ to be i.i.d., independent of $X_1, X_2, \dots$, and such that
\[
\mathbb{P}(\mathfrak{X}_k \in U) = \mathbb{P}(X_k \in U \mid X_k > 0) \text{ for } k \geq 2.
\]
Let
\[
\mathfrak{S}_n = \mathfrak{X}_2 + \dots + \mathfrak{X}_{n+1} \text{ for } n \geq 1
\]
and, for $L>0$, define the quantity
\[
K^\ast = \sup\{ k \geq 1 : \mathbb{P}(X_1 + \mathfrak{S}_k \leq L) > 0 \}.
\]
(We will assume that $L$ is large enough so that this set is nonempty, so $K^\ast$ is defined.) In the parity example from last section, $\mathfrak{X}_k \equiv 2$ for $k \geq 2$ and $K^\ast = \lfloor L/2 \rfloor$. Our last assumption is that 
\[
I_2 = 0,
\]
where $I_k$ was defined in the last section as the infimum of the support of $X_k$. In the absence of this assumption, the event $\{S_n \leq L\}$ is empty for large $n$.

\begin{Th}\label{thm: general_parity}
Suppose that $p \in (0,1]$, $I_2=0$, and $L$ is large enough that $\mathbb{P}( X_1 + \mathfrak{X}_2 \leq L)>0$. The following statements hold.
\begin{enumerate}
\item If $K^\ast = \infty$, then for every $\delta>0$,
\begin{equation}\label{eq: comprehensible_input}
\lim_{n \to \infty} \mathbb{P}(X_1 > I_1 + \delta \mid S_n \leq L) = 0.
\end{equation}
\item If $K^\ast < \infty$, then for every $\delta>0$,
\begin{equation}\label{eq: comprehensible_output}
\lim_{n \to \infty} \mathbb{P}(X_1 > I_1 + \delta \mid S_n \leq L) = \mathbb{P}(X_1 > I_1 + \delta \mid X_1 + \mathfrak{S}_{K^\ast} \leq L).
\end{equation}
\end{enumerate}
\end{Th}

\begin{proof}
If $K^\ast = \infty$, then
\eqref{eq: sum_diverges_condition} holds. Indeed, we have $\sum_{k=1}^\infty I_k = I_1 <\infty$ and if $\delta'>0$, then choosing $\eta$ such that $\mathbb{P}(X_2 \in [\eta,\delta'))>0$ (which is possible since $I_2=0$ and $K^\ast = \infty$), we find
\[
\sum_{k=1}^\infty \mathbb{E}X_k\mathbf{1}_{\{X_k \leq \delta'\}} \geq \eta \sum_{k=1}^\infty \mathbb{P}(X_k \in [\eta, \delta')) =\infty.
\]
Since \eqref{eq: sum_diverges_condition} implies \eqref{eq: comprehensible_input}, this completes the proof of \eqref{eq: comprehensible_input}.

If $K^\ast <\infty$, then because $I_2=0$, we must have $p<1$. 
To prove \eqref{eq: comprehensible_output}, we condition on which of $X_2, \dots, X_{n+1}$ are positive. Putting $S_n' = X_2 + \dots + X_{n+1}$ and letting $U \subset \mathbb{R}$ be Borel, we can use independence to derive
\begin{align*}
&\mathbb{P}(X_1 \in U, S_{n+1} \leq L)  \\
=~& \mathbb{P}(X_1 \in U, X_1 + S_n' \leq L) \\
=~&\sum_{k=0}^n \sum_{\#J=k} \mathbb{P}\left(X_1 \in U, X_1 + \sum_{i \in J} X_i \leq L, X_i > 0 ~\forall i \in J, X_i=0 ~\forall i \in \{2, \dots, n+1\} \setminus J\right) \\
=~& \sum_{k=0}^n \sum_{\#J=k} p^k (1-p)^{n-k} \mathbb{P}\left(X_1 \in U, X_1 + \sum_{i \in J} X_i \leq L \mid X_i > 0 ~\forall i \in J\right)  \\
=~& \sum_{k=0}^n  \binom{n}{k} p^k (1-p)^{n-k} \mathbb{P}\left(X_1 \in U, X_1 + \mathfrak{S}_k \leq L\right).
\end{align*}
If $k > K^\ast$, the probability is zero. Therefore for $n \geq K^\ast$, we have
\begin{align*}
\mathbb{P}(X_1 > I_1 + \delta \mid S_{n+1} \leq L) &= \frac{\sum_{k=0}^{K^\ast} \binom{n}{k} p^k (1-p)^{n-k} \mathbb{P}\left(X_1 > I_1 + \delta, X_1 + \mathfrak{S}_k \leq L\right)}{\sum_{k=0}^{K^\ast} \binom{n}{k} p^k (1-p)^{n-k} \mathbb{P}\left(X_1 + \mathfrak{S}_k \leq L\right)} \\
&= \frac{\sum_{k=0}^{K^\ast} \binom{n}{k} p^k (1-p)^{-k} \mathbb{P}\left(X_1 > I_1 + \delta, X_1 + \mathfrak{S}_k \leq L\right)}{\sum_{k=0}^{K^\ast} \binom{n}{k} p^k (1-p)^{-k} \mathbb{P}\left(X_1 + \mathfrak{S}_k \leq L\right)}.
\end{align*}
Letting $n \to \infty$, we obtain $\mathbb{P}(X_1 > I_1 + \delta \mid X_1 + \mathfrak{S}_{K^\ast} \leq L)$ and this completes the proof of \eqref{eq: comprehensible_output}.
\end{proof}

Although the above assumptions require that $X_2, X_3, \dots$ are i.i.d, Theorem~\ref{thm: general_parity} can be applied in some form to more general sequences. For example, if $X_1, X_2, \dots$ are nonnegative and independent and $X_k, X_{k+1}, \dots$ are i.i.d.~for some $k \geq 2$, then set $Y_1 = X_1 + \dots +X_{k-1}$ and $Y_\ell = X_{k+\ell-2}$ for $\ell \geq 2$. The sequence $(Y_n)$ now satisfies the conditions of Theorem~\ref{thm: general_parity}. Consequently, we obtain a formula for the limit of $\mathbb{P}(X_1 + \dots + X_{k-1} > I_1 + \dots +I_{k-1} + \delta \mid S_n \leq L)$. In particular, if $(X_n)$ already satisfies the conditions of the theorem, and we define $(Y_n)$ as above for a fixed $k \geq 2$, we can determine whether the (conditional) joint distribution of $(X_1, \dots, X_{k-1})$ converges to $\prod_{j=1}^{k-1} \delta_{\{I_j\}}$.

Returning to the parity example from the previous section, we have $K^\ast = \lfloor L/2\rfloor$ and therefore for $\delta>0$,
\[
\lim_{n \to \infty} \mathbb{P}(X_1 > \delta \mid S_n \leq L) = \mathbb{P}(X_1 >  \delta \mid X_1 + \mathfrak{S}_{K^\ast} \leq L) = \mathbb{P}\left(X_1 = 1 \mid X_1 + 2 \left\lfloor \frac{L}{2}  \right\rfloor \leq L \right).
\]
If $L$ is even, then $2\lfloor L/2 \rfloor = L$ and so the probability is zero. If $L$ is odd, then $2\lfloor L/2\rfloor = L-1$ and the probability is $\mathbb{P}(X_1=1)$, in accordance with the previous results.

As another example, suppose $X_1$ takes values $0$ or $1$ with $\mathbb{P}(X_1=0)>0$, and $X_2, X_3, \dots$ are i.i.d.~with $I_2=0$, $\mathbb{P}(X_2=I_2)>0$, and $\mathbb{P}(X_2=J_2)=0$, where $J_2 =\inf((\text{supp}~X_2)\setminus \{0\})$. For instance, $X_1$ could be a Bernoulli variable and $X_2$ could be $0$ with probability $q$ and uniform $[a,b]$ with probability $1-q$. In this case, one can show from formula \eqref{eq: comprehensible_input} that $\mathbb{P}(X_1>\delta \mid S_n \leq L)$ converges to 0 for all $\delta \in (0,1)$ if and only if $\lceil L/J_2 \rceil > \lceil (L-1)/J_2 \rceil$. 

Formula \eqref{eq: comprehensible_input} gives an exact limit, so it can be used to determine whether the distribution of $X_1$ is trivial in the limit (conditionally) for all examples of the type considered in this section. The answer depends on the infima of the supports of the variables and whether there are point masses at these values, and does not take a simple form.

\subsection{Partition examples}

The next examples are related to random partitions of integers. Let $(X_n)_{n \geq 1}$ be an independent sequence of random variables such that
\begin{equation}\label{eq: partition_distribution}
\mathbb{P}(X_n = 0) = \frac{1}{2} = \mathbb{P}(X_n = n).
\end{equation}
We will prove that
\begin{equation}\label{eq: partition_to_prove}
\liminf_{L \to \infty} \lim_{n \to \infty} \mathbb{P}(X_1 = 1 \mid X_1 + \dots + X_n \leq L) > 0,
\end{equation}
through two different proofs. For notational simplicity, we take $L$ to be an integer.

It is interesting to note that, although the distribution in \eqref{eq: partition_distribution} gives a nontrivial limit as in \eqref{eq: partition_to_prove}, a small perturbation of these variables can give a trivial limit. Namely, letting $\epsilon>0$, and letting $(X_n)_{n \geq 1}$ be an independent sequence of random variables satisfying $\mathbb{P}(X_n=0) = 1/2 - \epsilon = \mathbb{P}(X_n = n)$, but $\mathbb{P}(X_n=1)=2\epsilon$, we can apply the condition in Remark~\ref{rem: gap_remark} to see that $\mathbb{P}(X_1=1 \mid X_1 + \dots + X_n \leq L) \to 0$ as $n \to \infty$ for each $L$.

\bigskip
\noindent
{\bf Partition counting.} For $n \geq L \geq 1$, we have
\begin{align}
\mathbb{P}(X_1 = 1 \mid X_1 + \dots + X_n \leq L) &= \frac{\mathbb{P}(X_1=1) \mathbb{P}(X_2 + \dots + X_n \leq L-1)}{\mathbb{P}(X_1 + \dots + X_n \leq L)} \nonumber \\
&\geq \mathbb{P}(X_1=1) \frac{\mathbb{P}(X_1 + \dots + X_n \leq L-1)}{\mathbb{P}(X_1 + \dots + X_n \leq L)}.\label{eq: initial_partition_counting}
\end{align}
For any integer $m \geq 1$, the value of $\mathbb{P}(X_1 + \dots + X_n \leq m)$ is exactly the probability that the sum of the elements in a uniformly selected random subset of $\{1, \dots, n\}$ is at most $m$. For $n \geq m$, the number of such subsets whose sum is at most $m$ can be expressed as $\sum_{k=0}^m q(k)$, where $q(k)$ is the number of unique integer partitions of $k$. Here, ``unique'' refers to the restriction that the integers forming the partition appear without repitition. Therefore, we have
\[
\mathbb{P}(X_1 + \dots + X_n \leq m) = \frac{1}{2^n} \sum_{k=0}^m q(k).
\]
We combine this with \eqref{eq: initial_partition_counting} to obtain for $n \geq L$
\begin{equation}\label{eq: partition_counting_2}
\mathbb{P}(X_1 = 1 \mid X_1 + \dots + X_n \leq L) \geq \mathbb{P}(X_1=1) \frac{\sum_{k=0}^{L-1} q(k)}{\sum_{k=0}^{L-1} q(k) + q(L)}.
\end{equation}
A standard fact (see \cite[p.~69]{A76}) about integer partitions is that,
\[
q(k) \cdot  \frac{4 \cdot 3^{\frac{1}{4}} k^{\frac{3}{4}}}{e^{\pi \sqrt{\frac{k}{3}}}} \to 1 \text{ as } k \to \infty.
\]
From this one can show that
\[
\frac{q(L)}{\sum_{k=0}^{L-1} q(k)} \to 0 \text{ as } L \to \infty.
\]
Combining this statement with \eqref{eq: partition_counting_2}, we obtain
\begin{equation}\label{eq: partition_lower_bound}
\liminf_{L \to \infty} \lim_{n \to \infty} \mathbb{P}(X_1=1 \mid X_1 + \dots + X_n \leq L) \geq \mathbb{P}(X_1=1),
\end{equation}
and since $\mathbb{P}(X_1=1) = 1/2$, this shows \eqref{eq: partition_to_prove}.

We can strengthen \eqref{eq: partition_to_prove} using the Harris inequality. We observe that since the event $\{X_1 + \dots + X_n \leq L\}$ is decreasing in the variables $(X_n)$ and $\{X_1 = 1\} = \{X_1 \geq 1\}$ is increasing, the Harris inequality gives
\begin{equation}\label{eq: harris_upper_bound}
\mathbb{P}(X_1 = 1 \mid X_1 + \dots + X_n \leq L) \leq \mathbb{P}(X_1=1).
\end{equation}
Along with \eqref{eq: partition_lower_bound}, this proves 
\begin{equation}\label{eq: stronger_statement}
\lim_{L \to \infty} \lim_{n \to \infty} \mathbb{P}(X_1 = 1 \mid X_1 + \dots + X_n \leq L) = \mathbb{P}(X_1=1).
\end{equation}

\bigskip
\noindent
{\bf Injective mapping.}

The second proof for \eqref{eq: partition_to_prove} uses an injective mapping. It is a bit more general, so we allow to select $p \in (0,1)$ and let $(X_n)_{n \geq 1}$ be an independent sequence of random variables satisfying
\[
\mathbb{P}(X_n = 0) = p = 1-\mathbb{P}(X_n = n) \text{ for all } n \geq 1.
\]
Again, for $n \geq L \geq 1$, the Harris inequality gives \eqref{eq: harris_upper_bound}. For the lower bound, we write
\begin{align}
&\mathbb{P}(X_1 = 1 \mid X_1 + \dots + X_n \leq L) \nonumber \\
=~& \frac{\mathbb{P}(X_1=1) \mathbb{P}(X_2 + \dots + X_n \leq L-1)}{\mathbb{P}(X_1 + \dots + X_n \leq L)} \nonumber \\
=~& (1-p) \frac{\mathbb{P}(X_2 + \dots + X_n \leq L-1)}{\mathbb{P}(X_2 + \dots + X_n \leq L-1) + p \mathbb{P}(X_2 + \dots + X_n = L)}. \label{eq: harpers_decomposition}
\end{align}

We now prove that 
\begin{equation}\label{eq: to_prove_injective}
\mathbb{P}(X_2 + \dots + X_n = L) \leq \frac{1-p}{p} \mathbb{P}(X_2 + \dots + X_n \leq L-1).
\end{equation}
To do this, we decompose
\[
\mathbb{P}(X_2 + \dots + X_n = L) = \sum_{r=1}^L \left( p^{n-r}(1-p)^r  \times \#\mathcal{G}_r \right),
\]
where $\mathcal{G}_r = \left\{J \subset \{2, \dots, L\} : \# J = r, ~\sum_{i \in J} i = L\right\}$. Letting $\mathcal{H}_{r-1}$ be the collection 
\[
\mathcal{H}_{r-1} = \left\{J \subset \{2, \dots, L\} : \# J = r-1,~ \sum_{i\in J} i \leq L-1\right\},
\]
we define the function $\phi : \mathcal{G}_r \to \mathcal{H}_{r-1}$ by
\[
\phi(J) = J \setminus \{\min J\}.
\]
Then $\phi$ is injective, and this implies that $\# \mathcal{G}_r \leq \# \mathcal{H}_{r-1}$, so
\begin{align*}
\mathbb{P}(X_2 + \dots + X_n = L) \leq \sum_{r=1}^L \left( p^{n-r}(1-p)^r \times \# \mathcal{H}_{r-1}\right) &= \frac{1-p}{p} \sum_{r=0}^{L-1} \left( p^{n-r}(1-p)^{r} \times \# \mathcal{H}_r\right) \\
&= \frac{1-p}{p} \mathbb{P}\left( X_2 + \dots + X_n \leq L-1\right).
\end{align*}
This is \eqref{eq: to_prove_injective}.

Returning to \eqref{eq: harpers_decomposition} with \eqref{eq: to_prove_injective} in hand, we produce
\[
\mathbb{P}(X_1 = 1 \mid X_1 + \dots + X_n \leq L)  = (1-p) \frac{1}{1 + p \frac{\mathbb{P}(X_2 + \dots + X_n = L)}{\mathbb{P}\left( X_2 + \dots + X_n \leq L-1\right)}} \geq \frac{1-p}{1+1-p} = \frac{1-p}{2-p}.
\]
This shows \eqref{eq: partition_to_prove}. 

It is possible to improve this argument to give the stronger statement \eqref{eq: stronger_statement}. Again, the Harris inequality gives \eqref{eq: harris_upper_bound}, so one only needs to show \eqref{eq: partition_lower_bound}. This can be done by modifying the map $\phi$ to a ``multi-valued map,'' which associates $r$ many images in $\mathcal{H}_{r-1}$ to each $J \in \mathcal{G}_r$. Namely, if $J \in \mathcal{G}_r$ is written as $J = \{x_1, \dots, x_r\}$, then each set $J \setminus \{x_i\}$ for $i=1, \dots, r$ is an element of $\mathcal{H}_{r-1}$. This multi-valued map is injective in the sense that no distinct elements of $\mathcal{G}_r$ are associated to a common element in $\mathcal{H}_{r-1}$. Therefore $\#\mathcal{G}_r \leq r^{-1} \#\mathcal{H}_{r-1}$, and this improved estimate leads to \eqref{eq: partition_lower_bound}. A variant of this method can also handle some sequences $(X_n)$ with $\mathbb{P}(X_n = 0) = p = 1-\mathbb{P}(X_n = k_n)$, where the $k_n$ are integers that do not grow too slowly. 

\bigskip
\noindent
{\bf Criteria for nontrivial limits}

We now generalize the distribution considered in \eqref{eq: partition_distribution}. For a sequence $(a_n)_{n \geq 1}$ of positive integers such that $a_1=1$, we let $(X_n)_{n \geq 1}$ be independent random variables satisfying
\begin{equation}\label{eq: jacks_distributions}
\mathbb{P}(X_n = a_n) = \frac{1}{2} = \mathbb{P}(X_n = 0).
\end{equation}
Letting $\alpha_n = \#\{i : a_i = n\}$, we will suppose that $\alpha_n < \infty$ for all $n$, and so, without loss in generality, we will also suppose that $a_1 \leq a_2 \leq a_3 \leq \dots.$ In this setting, we have the following classification result. It states that if $\alpha_n^{1/n}$ is bounded in $n$, then there is a subsequence on which the conditional distribution of $X_1=1$ given $X_1 + \dots + X_n \leq L$ tends to a nontrivial limit. On the other hand, if $\alpha_n^{1/n}$ is not bounded, there is a subsequence on which the conditional distribution tends to a trivial limit.
\begin{Th}\label{thm: jacks_delight}
The following hold for independent $(X_n)_{n \geq 1}$ satisfying \eqref{eq: jacks_distributions}.
\begin{enumerate}
\item If $\limsup_{n \to \infty} \alpha_n^{1/n} = \infty$, then
\[
\liminf_{L \to \infty} \lim_{n \to \infty} \mathbb{P}(X_1=1 \mid X_1 + \dots + X_n \leq L) = 0.
\]
\item If $\limsup_{n \to \infty} \alpha_n^{1/n} < \infty$, then
\[
\limsup_{L \to \infty} \lim_{n \to \infty} \mathbb{P}(X_1=1 \mid X_1 + \dots + X_n \leq L) > 0.
\]
\end{enumerate}
\end{Th}
\begin{proof}
We begin with a couple of definitions. Set
\[
r_n = \sum_{i=1}^n \alpha_i,
\]
and for a positive integer $L$, we define
\[
q(L) = \#\{(v_1, \dots, v_L) \in \{0,1\}^{r_L} : a_1v_1 + \dots + a_{r_L}v_{r_L} = L\}
\]
and $Q(L) = q(1) + \dots + q(L)$. Here we have defined $r_n$ as the number of $a_i$'s that are $\leq n$, and $q(L)$ counts a kind of partition of $L$ with multiplicity. It is the number of ways to represent $L$ as the sum of the integers $a_1, a_2, \dots$, which may not be distinct. Concerning these definitions, we observe that
\begin{equation}\label{eq: jacks_delight_1}
\mathbb{P}(X_1 + \dots + X_n = L) = 2^{-n} q(L) \text{ for } n \geq r_L.
\end{equation}
This is because the left side is
\[
\mathbb{P}(X_1 + \dots + X_{r_L} = L, ~X_{r_L+1} = \dots = X_n = 0) = 2^{-n+r_L} 2^{-r_L} q(L).
\]

We claim that for $n \geq r_L$,
\begin{align}
\frac{Q(L-1)}{Q(L)} \mathbb{P}(X_1=0) &\leq \mathbb{P}(X_1=1 \mid X_1 + \dots + X_n \leq L) \nonumber \\
&= \mathbb{P}(X_1 = 0 \mid X_1 + \dots + X_n \leq L) \frac{Q(L-1)}{Q(L)}. \label{eq: white_bowl}
\end{align}
To prove this, observe that there is a natural bijection between the events $\{X_1 = 1, X_1 + \dots + X_n \leq L\}$ and $\{X_1=0, X_1 + \dots + X_n \leq L-1\}$ obtained by flipping the value of $X_1$ from $1$ to $0$. Therefore
\begin{equation}\label{eq: white_bowl_3}
\mathbb{P}(X_1=1, X_1 + \dots + X_n \leq L) = \mathbb{P}(X_1=0, X_1 + \dots + X_n \leq L-1).
\end{equation}
Because \eqref{eq: jacks_delight_1} implies 
\begin{equation}\label{eq: white_bowl_2}
\frac{\mathbb{P}(X_1 + \dots + X_n \leq L-1)}{\mathbb{P}(X_1 + \dots + X_n \leq L)} = \frac{Q(L-1)}{Q(L)},
\end{equation}
this shows the equality in \eqref{eq: white_bowl}. The inequality of \eqref{eq: white_bowl} follows from the Harris inequality because if we apply it in \eqref{eq: white_bowl_3}, we obtain
\[
\mathbb{P}(X_1 = 1, X_1 + \dots + X_n \leq L) \geq \mathbb{P}(X_1=0) \mathbb{P}(X_1 + \dots + X_n \leq L-1)
\]
and then we can use \eqref{eq: white_bowl_2} to get the inequality.

The display \eqref{eq: white_bowl} readily implies
\begin{equation}\label{eq: green_bowl}
\frac{1}{2} \cdot \frac{Q(L-1)}{Q(L)} \leq \lim_{n \to \infty} \mathbb{P}(X_1=1 \mid X_1 + \dots + X_n \leq L) \leq \frac{Q(L-1)}{Q(L)}.
\end{equation}
However it is elementary that
\begin{equation}\label{eq: green_bowl_2}
\liminf_{L \to \infty} \frac{Q(L)}{Q(L-1)} \leq \liminf_{L \to \infty} Q(L)^{\frac{1}{L}} \leq \limsup_{L \to \infty} Q(L)^{\frac{1}{L}} \leq \limsup_{L \to \infty} \frac{Q(L)}{Q(L-1)}
\end{equation}
and furthermore, we have
\begin{equation}\label{eq: green_bowl_3}
\limsup_{L \to \infty} Q(L)^{\frac{1}{L}} < \infty \text{ if and only if } \limsup_{L \to \infty} q(L)^{\frac{1}{L}} < \infty.
\end{equation}
This latter statement holds because $Q(L) \geq q(L)$, but if $q(L)^{1/L}$ is bounded by $C \geq 1$, then $Q(L) \leq LC^L$, giving that $Q(L)^{1/L}$ is also bounded.

We combine \eqref{eq: green_bowl}, \eqref{eq: green_bowl_2}, and \eqref{eq: green_bowl_3} to obtain the following.
\begin{enumerate}
\item If $\limsup_{L \to \infty} q(L)^{1/L} = \infty$, then
\[
\liminf_{L \to \infty} \lim_{n \to \infty} \mathbb{P}(X_1=1 \mid X_1 + \dots + X_n \leq L) = 0.
\]
\item If $\limsup_{L \to \infty} q(L)^{1/L} < \infty$, then
\[
\limsup_{L \to \infty} \lim_{n \to \infty} \mathbb{P}(X_1=1 \mid X_1 + \dots + X_n \leq L) > 0.
\]
\end{enumerate}
It remains to show that 
\begin{equation}\label{eq: last_to_show_bowl}
q(L)^{1/L} \text{ is bounded in }L \text{ if and only if }\alpha_n^{1/n} \text{ is bounded in }n, 
\end{equation}
and this follows from well-known methods in analytic number theory. To summarize, defining the power series
\[
f(z) = 1+\sum_{L=1}^\infty q(L)z^L = \prod_{n=1}^\infty (1+z^n)^{\alpha_n},
\]
we have the representation
\[
\log f(z) = \sum_{n=1}^\infty \alpha_n \log(1+z^n).
\]
(See \cite[Sec.~I.2.2]{FS08}.) In addition, the radius of convergence for $f$ (which is $\left( \limsup_{L \to \infty} q(L)^{1/L}\right)^{-1}$) is nonzero if and only if the radius of convergence for $\log f$ is nonzero (although these radii need not coincide). By expansion of $\log (1+z^n)$, the series for $\log f(z)$ has nonzero radius of convergence if and only if $\limsup_{n \to \infty} \alpha_n^{1/n} < \infty$.

\end{proof}

\subsection{Concluding remarks}

The above examples show that the limiting distribution from Question~\ref{question: sequence_question} can be trivial or non-trivial depending on the distribution of the variables $(X_n)$. This diverse range of behavior suggests there may not be a simple answer to the question. In fact, there are even sequences for which the limit does not exist. One example is given by an independent sequence $(X_n)_{n \geq 1}$ with $\mathbb{P}(X_n = 0) = 1/2 = \mathbb{P}(X_n = a_n)$, where $a_n$ is chosen as
\[
a_n = \begin{cases}
1 &\quad\text{if } n=1 \\
2 &\quad\text{if } r_{2k} \leq n < r_{2k+1} \text{ for some } k \geq 0 \\
3 &\quad\text{if } r_{2k+1} \leq n < r_{2k+2} \text{ for some } k \geq 0,
\end{cases}
\]
where $2 = r_0 < r_1 < r_2 < \dots$ is a sequence of integers. If $(r_k)$ grows rapidly enough, then the limit
\[
\lim_{n \to \infty} \mathbb{P}(X_1=1 \mid X_1 + \dots + X_n \leq 4) \text{ does not exist.}
\]

\end{document}